\documentclass[de, en, amspaper, secnum, nobiblatex]{./cls/myamspaper}

\usepackage[top=40mm, bottom=40mm, left=30mm, right=30mm]{geometry}


\DeclareMathAlphabet{\mathpzc}{OT1}{pzc}{m}{it}

\usepackage{marginnote}

\title[A Halmos-von Neumann theorem for actions of general groups]{A Halmos-von Neumann theorem for actions of general groups}

\begin{document}

\begin{abstract}
 	We give a new categorical approach to the Halmos-von Neumann theorem for actions of general topological groups. As a first step, we establish that the categories of topological and measure-preserving irreducible systems with discrete spectrum are equivalent. This allows to prove the Halmos-von Neumann theorem in the framework of topological dynamics. We then use the Pontryagin and Tannaka-Krein duality theories to obtain classification results for topological and then measure-preserving systems with discrete spectrum. As a byproduct, we obtain a complete isomorphism invariant for compactifications of a fixed topological group.\medskip
\end{abstract}

\maketitle

\section*{Introduction}

The desire to find meaningful isomorphism invariants for objects is present in many mathematical areas and has lead to fruitful new disciplines such as homological algebra, algebraic topology and K-theory. In ergodic theory it inspired the famous \emph{isomorphism problem}: Given a measure-preserving transformation $\varphi \colon \mathrm{X} \rightarrow \mathrm{X}$ on a probability space $\mathrm{X}$, can we find a simple isomorphism invariant which is complete, i.e., classifies $(\mathrm{X},\varphi)$ up to an isomorphism?\footnote{An account of the origins and history of the problem is given in \cite{ReWe2012}, see also \cite[Subsection 18.4.7]{EFHN2015}.}

For transformations which are \enquote{irreducible} (i.e., \emph{ergodic}) and \enquote{structured} (i.e., \emph{have discrete spectrum}) this question was first answered positively by P. Halmos and J. von Neumann in \cite{HaNe1942} (see also Section 4 of von Neumann's article \cite{vNeum1932}). Their results classify such systems in terms of the point spectrum $\sigma_{\mathrm{p}}(T_\varphi)$ of the induced Koopman operator\footnote{The original result considers the Koopman operator on $\mathrm{L}^2$-spaces. However, the Hilbert space structure is inessential for the result and the point spectrum of the Koopman operator coincides on all $\mathrm{L}^p$-spaces for $p \in [1, \infty]$, see \cite[Proposition 2.18]{EFHN2015}. We prefer $\mathrm{L}^1(\uX)$.} $T_\varphi \in \mathscr{L}(\mathrm{L}^1(\uX))$ defined by $T_\varphi f \coloneqq f \circ \varphi$ for $f \in \mathrm{L}^1(\uX)$ which is (for such systems) a subgroup of the complex unit circle $\T$.
In modern language their results can be summarized as follows (see \cite[Chapter 17]{EFHN2015}).

\begin{theorem*}
	The following assertions hold for ergodic systems with discrete spectrum.
		\begin{enumerate}[(i)]
			\item Two systems $(\mathrm{X}_1,\varphi_1)$ and $(\mathrm{X}_2,\varphi_2)$ are isomorphic if and only if $\sigma_{\mathrm{p}}(T_{\varphi_1}) = \sigma_{\mathrm{p}}(T_{\varphi_2})$.
			\item Given any subgroup $\sigma \subset \T$ there is a system $(\mathrm{X},\varphi)$ with $\sigma_{\mathrm{p}}(T_\varphi) = \sigma$.
			\item Any such system is isomorphic to a rotation system on a compact monothetic group.
		\end{enumerate}
\end{theorem*}

Thus, their result not only shows that the point spectrum of the Koopman operator is a complete isomorphism invariant, but also characterizes the resulting point spectra precisely as the subgroubs of the complex unit circle, and exhibits a canonical representative for every isomorphism class. These three aspects of the classification result are known as \emph{uniqueness}, \emph{realization} and \emph{representation}. 

The theorem has been modified and generalized in several directions. There is a version for minimal homeomorphisms on compact spaces with discrete spectrum (see \cite[Section 5.5]{Walters1975}, \cite[Chapter VIII]{DNP1987}, and \cite[Theorem 3.1 and Section 4]{Edeko2019}) as well as an extension to topological and measure-preserving systems with quasi-discrete spectrum (see \cite{Abra1962}, \cite{HaPa1965}, and \cite{HaMo2018}),
 and to non-ergodic systems (see \cite{Kwia1981} and \cite{Edeko2019}). Moreover, there are versions for actions of abelian groups (see \cite[Subsection 1.5]{JaShTa2021} and \cite[Theorem 2.6]{Stor1974}), even on non-commutative von Neumann algebras (see \cite[Sections 5 and 7]{OPT1980}). G. Mackey (see \cite{Mack1964}) and R. Ellis (in \cite{Elli1987}) gave a partial generalization of the original result to actions of arbitrary groups. The representation aspect is treated even for group representations on Banach lattices with discrete spectrum in a paper by R. Nagel and M. Wolff (see \cite{NaWo1972}).
 
A further remarkable step has been made by R. Zimmer in \cite{Zimm1976}: He proved a \enquote{full version} of the Halmos-von Neumann theorem covering all three aspects for so-called \emph{normal} actions of a second countable,  locally compact group on standard probability spaces. Here, normality is a strengthened version of discrete spectrum which arises when studying actions of non-abelian groups. Zimmer's results even cover the more general case of normal \emph{extensions} of ergodic measure-preserving systems. His proofs are very measure-theoretic in nature and some of his concepts rely on the rather involved theory of virtual groups.

One goal of this work is to provide a new and accessible proof of Zimmer's result in the case of systems by using topological models. With this approach we avoid any measure-theoretic intricacies by first proving a purely topological version of the result and then simply \enquote{translating} it to ergodic theory. This also allows to drop any separability assumptions imposed on the group and the spaces it acts on. Our work therefore contributes to a recent endeavour by A. Jamneshan and T. Tao to remove such assumptions from classical results of ergodic theory (see, e.g.,  \cite{JamnTao2019pre}, \cite{JamnTao2020apre}, \cite{JamnTao2020bpre}, \cite{Jamn2020pre}). 

The second major purpose of this article is to give a more concise and precise formulation of the coherences established by the Halmos-von Neumann theorem by using category theoretical language. Written in this way, the topological version of the Halmos-von Neumann theorem establishes an equivalence between three categories for a fixed topological group $G$: the category of (pointed) normal $G$-actions, the category of compactifications of $G$ and the category of \emph{grouplike subsets} of its dual (see \cref{main2} for the precise formulation). Moreover, we show that there is an equivalence between the categories of normal topological and measure-preserving systems, respectively (see \cref{normalequiv}). Thus, all these mathematical worlds are---in essence---the same.

This article is also a preparation for upcoming work in which we discuss structured extensions in topological dynamics and ergodic theory. Even if one is only interested in $\Z$-actions, more general (possibly non-commutative) groups arise naturally when studying extensions between systems having relative discrete spectrum. A thorough understanding of the Halmos-von Neumann theorem for actions of arbitrary groups is therefore the basis for studying the \enquote{relative case} of structured extensions.

Let us finally remark that the key results of this article are on compactifications of topological groups. In fact, the proof of the Halmos-von Neumann theorem is closely tied to the duality theory of compact groups. The well-known and elegant Pontryagin duality (see, e.g., \cite[Section 4.3]{Foll2016}) is already used in the original work of Halmos and von Neumann and can still be employed when dealing with actions of abelian groups (see \cref{scatcomm} below). In the non-commutative case, it is replaced by the much more sophisticated Tannaka-Krein duality (see, e.g., \cite[Paragraph 30]{HeRo1970} and \cite{JoSt1991}). As pointed out in the introduction of \cite[Paragraph 30]{HeRo1970}, a general compact group is not determined by its dual, i.e., the set of equivalence classes of irreducible representations. This is the reason why the Tannaka-Krein duality is rather difficult to state.  However, we show below that, when dealing with compatifications $H$ of a fixed topological group $G$, the situation becomes much easier: The dual of $H$ (embedded as a subset of the dual of $G$) already determines the compactification uniquely up to an isomorphism (see \cref{grpcompvsgrplike}). This purely harmonic analytic observation is interesting in its own right and is the essential part of our results.

\subsection*{Organization of the article and main results.} We now briefly outline the structure of this article. In Section \ref{sdiscspectrum} we recall the concept of discrete spectrum in the general framework of bounded strongly continuous group representations (see Subsection \ref{sdiscspectrumg}). We then discuss properties of topological (see Subsection \ref{sdiscspectrumt}) and measure-preserving (see Subsection \ref{sdiscspectrumm}) dynamical systems having discrete spectrum. In Section \ref{scattheory} we provide a short repetition of basic category theoretical language illustrated by means of the Gelfand representation theory for unital commutative C*-algebras. We establish in Section \ref{sequivalence} that the categories of irreducible topological and measure-preserving dynamical systems with discrete spectrum are equivalent (see \cref{equivtopmeas}). The Halmos-von Neumann theorem for actions of commutative groups is then discussed in Section \ref{sclassicalcom} (in a classical formulation, see \cref{comhvn1} and \cref{comhvn2}) and in Section \ref{scatcomm} (in terms of category theory, see \cref{hvncat1}).\footnote{The generalization of the classical ergodic theoretic result from $\Z$-actions to actions of general abelian groups is straightforward and probably quite well-known, see, e.g., \cite[Subsection 1.5]{JaShTa2021}. We have included it anyway in order to highlight our categorial perspective, to illustrate our topological approach and to motivate our methods in the general, non-commutative case.} Our main results are then contained in Section \ref{scatgen}. After introducing the necessary concepts in Subsections \ref{scatgen1}, \ref{scatgen2} and \ref{scatgen3}, we prove a Tannaka-Krein-type duality result for compactifications of topological groups (see \cref{grpcompvsgrplike}). After introducing normal systems in Subsection \ref{scatgen5}, we use this to prove versions of the Halmos-von Neumann theorem for actions of general groups in Subsection \ref{scatgen6}, first in a category theoretical language (see \cref{main2}) and then in a more classical formulation (see \cref{nomralhvntop2} and \cref{nomralhvntop2}).

\subsection*{Terminology and notation.} In the following all compact spaces are Hausdorff and all vector spaces are complex. If $E$ is a Banach space, we write $\mathscr{L}(E)$ for the space of bounded linear operators on $E$, and $E'$ for the dual space. 

We now briefly recall some concepts and terminology from topological dynamics and ergodic theory (see, e.g., \cite{Ausl1988}, \cite{EiWa2011}, \cite{OlVi2016}, and \cite{EFHN2015} for an introduction). Here and in the following, $G$ denotes a fixed topological group and $1$ is its neutral element.

A \emph{topological dynamical system} is a pair $(K,\varphi)$ consisting of a compact space $K$ and a continuous group action $\varphi \colon G \times K \rightarrow K, \, (t,x) \mapsto \varphi_t(x) = tx$. We call $(K,\varphi)$ \emph{minimal} if $K$ and $\emptyset$ are the only closed invariant subsets of $K$.
 
Topological dynamical systems can be studied effectively from an operator theoretic perspective via the induced \emph{Koopman representation}
	\begin{align*}
		T^\varphi \colon G \rightarrow \mathscr{L}(\mathrm{C}(K)), \quad t \mapsto T_t^\varphi
	\end{align*}
where $T^\varphi_tf \coloneqq f \circ \varphi_{t^{-1}}$ for $f \in \mathrm{C}(K)$. This is a strongly continuous group representation, i.e., the map $G \rightarrow \mathrm{C}(K), \, t \mapsto T_t^\varphi f$ is continuous for every $f \in \mathrm{C}(K)$ (see \cite[Theorem 4.17]{EFHN2015}).

We recall that a \emph{morphism} $q \colon (K,\varphi) \rightarrow (L,\psi)$ between topological dynamical systems $(K,\varphi)$ and $(L,\psi)$ is a continuous map $q\colon K \rightarrow L$ such that the diagram
	\[
		\xymatrix{
			K \ar[r]^{\varphi_t}  \ar[d]^q &K  \ar[d]^q\\
			L \ar[r]^{\psi_t} & L
		}
	\]
commutes for every $t \in G$. If $q$ is surjective (which is automatically the case whenever $(L,\psi)$ is minimal), then $q$ is also called an \emph{extension}. Every morphism $q \colon (K,\varphi) \rightarrow (L,\psi)$ induces an operator $J_q \in \mathscr{L}(\mathrm{C}(L) , \mathrm{C}(K))$ given by $J_qf \coloneqq f \circ q$ for $f \in \mathrm{C}(L)$ intertwining the group representations $T^\psi$ and $T^\varphi$. Moreover, $q$ is an extension if and only if $J_q$ is an isometry.

Classically, if $G$ is a discrete group, a \emph{measure-preserving point-transformation} is a pair $(\uX,\varphi)$ where $\mathrm{X} = (X, \Sigma_\uX,\mu_\uX)$ is a probability space and $\varphi = (\varphi_t)_{t \in G}$ is a family of measurable and measure-preserving maps on $\uX$ such that $\varphi_{1} = \mathrm{id}_X$ and $\varphi_{t_1t_2} = \varphi_{t_1}\circ \varphi_{t_2}$ almost everywhere for all $t_1,t_2 \in G$. It is called  \emph{ergodic} if each invariant subset has either measure $0$ or $1$.

As in the topological case, every measure-preserving system gives rise to a \emph{Koopman representation}
	\begin{align*}
		T^\varphi\colon G \rightarrow \mathscr{L}(\mathrm{L}^1(\uX)), \quad t \mapsto T_t,
	\end{align*}
where again $T_t^\varphi f \coloneqq f \circ \varphi_{t^{-1}}$ for $f \in \mathrm{L}^1(\uX)$ and $t \in G$. For every $t \in G$ the operator $T_t \coloneqq T^\varphi_t \in \mathscr{L}(\mathrm{L}^1(\uX))$ is a \emph{Markov isomorphism}, i.e., $T_t$ is an invertible isometry with $|T_t f|= T_t|f|$ for all $f \in \mathrm{L}^1(\uX)$ and $T_t \mathbbm{1} = \mathbbm{1}$ (see \cite[Section 13.2]{EFHN2015}). If $\uX$ is a standard probability space, then a result of von Neumann tells that any group representation $T \colon G \rightarrow \mathscr{L}(\mathrm{L}^1(\uX))$ of $G$ as Markov isomorphisms on $\mathrm{L}^1(\uX)$ is induced by a measure-preserving system $(\mathrm{X},\varphi)$ in this way. Moreover, $\varphi_t$ is unique up to a null-set for every $t \in G$ (see \cite[Theorem 7.20]{EFHN2015}). On general probability spaces such representations are still induced by transformations of the measure algebra (see \cite[Theorem 12.10]{EFHN2015}). However, we avoid these measure-theoretic intricacies and stick with the functional analytic perspective. Thus, here a \emph{measure-preserving system} (for an arbitrary topological group $G$) is a pair $(\mathrm{X},T)$ consisting of a probability space $\uX$  and a strongly continuous group representation $T \colon G \rightarrow \mathscr{L}(\mathrm{L}^1(\uX)), \, t \mapsto T_t$ as Markov lattice isomorphisms on $\mathrm{L}^1(\uX)$. It is \emph{ergodic} if the \emph{fixed space}
	\begin{align*}
		\fix(T) \coloneqq \{f \in \mathrm{L}^1(\uX)\mid T_tf = f \textrm{ for every } t \in G\}
	\end{align*}
is one-dimensional.

Finally, a \emph{morphism} (or \emph{extension}) $J \colon (\uY,S) \rightarrow (\uX,T)$ between such systems is a \emph{Markov embedding} $J \in \mathscr{L}(\mathrm{L}^1(\uY),\mathrm{L}^1(\uX))$ intertwining the representations, i.e., $J$ is an isometry, satisfies  $|Jf| = J|f|$ for all $f \in \mathrm{L}^1(\uX)$ as well as $J\mathbbm{1} = \mathbbm{1}$,  and the diagram
	\[
		\xymatrix{
			\mathrm{L}^1(\uX)  \ar[r]^{T_t}  &\mathrm{L}^1(\uX)  \\
			\mathrm{L}^1(\uY)  \ar[u]_J \ar[r]^{S_t} & \mathrm{L}^1(\uY)  \ar[u]_J  
		}
	\]
commutes for every $t \in G$.

Many examples in ergodic theory arise from topological dynamical systems. Given a topological dynamical system $(K,\varphi)$ we write $\mathrm{P}_{\varphi}(K)$ for the space of invariant probability measures\footnote{Here and in the following, a probability measure on a compact space is assumed to be a regular Borel measure.} on $K$. By the Riesz representation theorem we can identify this with a weak*-compact convex subset of the dual space $\mathrm{C}(K)'$ of $\mathrm{C}(K)$. We also remark that if $q \colon (K,\varphi) \rightarrow (L,\psi)$ is a morphism to another system $(L,\psi)$ and $\mu \in \mathrm{P}_{\varphi}(K)$, then $q_{*}\mu \coloneqq \mu \circ J_q \in \mathrm{P}_{\psi}(L)$. We call $q_{*} \mu$ the \emph{push-forward of $\mu$}. 

Now, if $\mu$ is an invariant probability measure of a system $(K,\varphi)$, we write $(K,\mu)$ for the arising Borel probability space and note that the Koopman representation $T^\varphi \colon G \rightarrow \mathscr{L}(\mathrm{C}(K))$ uniquely extends to a strongly continuous group representation 
	\begin{align*}
		T^\varphi \colon G \rightarrow \mathscr{L}(\mathrm{L}^1(K,\mu)), \quad t \mapsto T_t^{\varphi}
	\end{align*}
as Markov isomorphisms. Thus, every invariant probability measure $\mu$ of a topological dynamial system $(K,\varphi)$ gives rise to a measure-preserving system and we denote this by $(K,\mu,T^\varphi)$.

\section{Systems with discrete spectrum}\label{sdiscspectrum}

\subsection{Group representations with discrete spectrum}\label{sdiscspectrumg}
Structured systems in topological dynamics and ergodic theory can be introduced quite elegantly via the respective Koopman representations. We therefore first cover the notion of \emph{discrete spectrum} in the abstract, operator theoretic framework of representations $T \colon G \rightarrow \mathscr{L}(E), \, t \mapsto T_t$ of $G$ on a Banach space $E$ which are \emph{(norm) bounded} (i.e., $\sup_{t \in G} \|T_t\| < \infty$) and \emph{strongly continuous} (i.e., the map $G \rightarrow E, \, t \mapsto T_tv$ is continuous for every $v \in E$).

\begin{definition}
	A bounded strongly continuous group representation $T \colon G \rightarrow \mathscr{L}(E)$ on a Banach space $E$ \emph{has discrete spectrum} if
		\begin{align*}
			E = \overline{\bigcup \{M \subset E\mid M \textrm{ finite-dimensional invariant subspace}\}}.
		\end{align*}	
\end{definition}
Loosely speaking, this means that the representation is generated by its \enquote{small} subrepresentations. Here is a different characterization (extending \cite[Theorem 16.36]{EFHN2015}) in terms of the \emph{enveloping operator semigroup}
	\begin{align*}
		\mathcal{J}(T) \coloneqq \overline{\{T_t \mid t \in G\}} \subset \mathscr{L}(E)
	\end{align*}
where the closure is taken with respect to the strong operator topology (cf. \cite[Appendix C.8]{EFHN2015}).
\begin{proposition}\label{chardisc}
	For a bounded strongly continuous group representation $T \colon G \rightarrow \mathscr{L}(E)$ on a Banach space $E$ the following assertions are equivalent.
		\begin{enumerate}
			\item $T$ has discrete spectrum.
			\item The orbits $\{T_tv\mid t \in G\}$ are relatively compact for every $v \in E$.
			\item The semigroup $\mathcal{J}(T)$ is a compact topological group.\label{charc}
		\end{enumerate}
\end{proposition}
\begin{proof}
	Assume that (a) holds. Since bounded subsets of finite-dimensional normed spaces are relatively compact, we have a dense subset $D \subset E$ such that the orbit $\{T_tv \mid t \in G\}$ is relatively compact for every $v \in D$. But, again since the representation is bounded, this already implies that all orbits are relatively compact (see \cite[Exercise 16.10]{EFHN2015}), hence (b).
	
	If (b) holds, then $\mathcal{J}(T)$ is compact (see again \cite[Exercise 16.10]{EFHN2015}).  Since composing operators is jointly continuous on norm bounded subsets (see \cite[Proposition C.19]{EFHN2015}), we obtain that $\mathcal{J}(T)$ is a compact topological semigroup. Using compactness it is readily checked that if $S = \lim_{i} T_{t_i}$ with respect to the strong operator topology for some net $(T_{t_i})_{i \in I}$, then $(T_{t_i^{-1}})_{i \in I}$ converges to an inverse of $S$. A similar argument shows that  inversion is a continuous map on $\EuScript{J}$ and therefore $\mathcal{J}(T)$ is a compact topological group. This shows (c).
	
	Finally, if (c) holds, then (a) follows by a corollary of the Peter-Weyl representation theory for compact groups (see \cite[Theorem 15.14]{EFHN2015}).
\end{proof}
In the case of an abelian group $G$, we obtain the following additional characterization using the dual group
	\begin{align*}
		G^* \coloneqq \{\chi \colon G \rightarrow \T\mid \chi \textrm{ continuous group homomorphism}\}
	\end{align*}
of $G$ (where $\T$ is the complex unit circle).
\begin{proposition}\label{chardisc2}
	Assume that the group $G$ is commutative. A bounded strongly continuous group representation $T \colon G \rightarrow \mathscr{L}(E)$ on a Banach space $E$ has discrete spectrum if and only if
		\begin{align*}
			E = \overline{\lin} \{v \in E \mid \exists \chi \in G^* \textrm{ with } T_tv = \chi(t)v \textrm{ for every } t \in G\}.
		\end{align*}
\end{proposition}
\begin{proof}
	If $T$ has discrete spectrum, than the desired equality follows directly from \cref{chardisc} \ref{charc} combined with \cite[Corollary 15.18]{EFHN2015}. The other implication is trivial.
\end{proof}
\begin{remark}
	If $G= \Z$, we can identify $G^*$ with $\T$ and the condition in \cref{chardisc2} becomes the usual notion of discrete spectrum for power-bounded operators (cf. \cite[Section 16.36]{EFHN2015}).
\end{remark}

\subsection{Topological dynamical systems with discrete spectrum}\label{sdiscspectrumt} 
The definition of systems with discrete spectrum in topological dynamics is now a special case of the more general setting discussed in Subsection \ref{sdiscspectrumg}.
\begin{definition}
	A topological dynamical system $(K,\varphi)$ \emph{has discrete spectrum} if the induced Koopman representation $T^\varphi\colon G \rightarrow \mathscr{L}(\mathrm{C}(K))$ has discrete spectrum.
\end{definition}
Clearly, a trivial action on any compact space has discrete spectrum. Prototypical examples for minimal systems with discrete spectrum are the following \emph{rotations} and \emph{quasi-rotations}.

\begin{example}\label{topexample}
	Call a pair $(H,c)$ with $H$ a compact group and $c \colon G \rightarrow H$ a continuous group homomorphism with dense range a \emph{group compactification}\footnote{Note that we do not require $c$ to be a homeomorphism onto its range. In fact, we do not even demand that $c$ is injective.} of $G$ (see \cite[Remark D.12.3]{deVr1993}). Given such a group compactification $(H,c)$ the continuous group action
		\begin{align*}
			\varphi = \varphi_c\colon G \times H \rightarrow H, \quad (t,x) \mapsto c(t)x
		\end{align*}
	defines a minimal topological dynamical system $(H,\varphi_c)$ with discrete spectrum which we call a \emph{(topological) rotation} (use the representation theory of compact groups, see, e.g., \cite[Section 5.2]{Foll2016}), cf. \cite[Remark 3.35.6]{deVr1993}.
	
	Given a closed subgroup $U$ of $H$, we also obtain an action on the homogeneuous space $H/S$ via
		\begin{align*}
			\varphi = \varphi_{c,U} \colon G \times H/U \rightarrow H/U, \quad (t,xU) \mapsto (c(t)x)U.
		\end{align*}
	Then $(H/U,\varphi)$ is still a minimal topological dynamical system with discrete spectrum (as both properties are stable with respect to taking factors). We call this a \emph{(topological) quasi-rotation}.
\end{example}

There are additional well-known characterizations of topological dynamical systems with discrete spectrum. To recall them, we remind the reader of some standard concepts from topological dynamics (cf. \cite[Chapter 2]{Ausl1988}, \cite[Section IV.2]{deVr1993} and \cite[Definition 1.16]{EdKr2020}).

\begin{definition}\label{structureddef}
	A topological dynamical system $(K,\varphi)$ is 
		\begin{enumerate}[(i)]
			\item \emph{equicontinuous} if the set $\{\varphi_t\mid t \in G\} \subset \mathrm{C}(K,K)$ is equicontinuous.
			\item \emph{pseudoisometric} if there is a family $P$ of pseudometrics on $K$  generating the topology such that $p(tx,ty) = p(x,y)$ for all $x,y \in K$, $t \in G$ and $p \in P$.
			\item \emph{isometric} if $P$ in (ii) can be chosen to consist of a single element (which is then necessarily a metric on $K$).
		\end{enumerate}
\end{definition}

We also remind the reader of the construction of enveloping semigroups in topological dynamics (cf. \cite[Chapters 3 and 5]{Ausl1988}). We will use the following version (see also the introduction of \cite{EdKr2020}).

\begin{definition}\label{uniformenv}
	For a topological dynamical system $(K,\varphi)$ the closure
		\begin{align*}
			\mathrm{E}_{\mathrm{u}}(K,\varphi) \coloneqq \overline{\{\varphi_t \colon K \rightarrow K \mid t \in G\}} \subset \mathrm{C}(K,K)
		\end{align*}
	with respect to the topology of uniform convergence equipped with the composition of maps is the \emph{uniform enveloping semigroup of $(K,\varphi)$}.
\end{definition}

The following result shows that we can identify the uniform enveloping semigroup with the enveloping operator semigroup of the induced Koopman representations. However, it is sometimes convenient to work with one or the other.
\begin{proposition}\label{semigroups}
	Let $(K,\varphi)$ be a topological dynamical system. Then the map
		\begin{align*}
			\mathrm{E}_{\mathrm{u}}(K,\varphi) \rightarrow \mathcal{J}(T^\varphi), \quad \vartheta \mapsto S_\vartheta
		\end{align*}
	with $S_\vartheta f \coloneqq f \circ \vartheta$ for $f \in \mathrm{C}(K)$ and $\vartheta \in \mathrm{E}_{\mathrm{u}}(K,\varphi)$ is an anti-isomorphism of topological semigroups.
\end{proposition}
\cref{semigroups} is a direct consequence of the following lemma.
\begin{lemma}
	Let $K$ be a compact space. Equip $\mathrm{C}(K,K)$ with the topology of uniform convergence and the space of unital algebra homomorphisms $\mathrm{Hom}(\mathrm{C}(K))$ with the strong operator topology. Then the map
		\begin{align*}
			\mathrm{C}(K,K) \rightarrow \mathrm{Hom}(\mathrm{C}(K)), \quad \vartheta \mapsto S_\vartheta
		\end{align*}
	with $S_\vartheta f \coloneqq f \circ \vartheta$ for $f \in \mathrm{C}(K)$ and $\vartheta \in \mathrm{C}(K,K)$ is an anti-isomorphism of topological semigroups.
\end{lemma}
\begin{proof}
	The map is a bijection by \cite[Theorem 4.13]{EFHN2015} and it is clear that it is an antihomomorphism of semigroups. To show that it is a homeomorphism, let $(\vartheta_i)_{i \in I}$ be a net in $\mathrm{C}(K,K)$ and $\vartheta \in \mathrm{C}(K,K)$. Since the uniform structure of $K$ is generated by the continuous functions $f \in \mathrm{C}(K)$, we obtain that $\lim_{i} \vartheta_i = \vartheta$ uniformly if and only if
		\begin{align*}
			\lim_i S_{\vartheta_i} f = \lim_i f\circ \vartheta_i = f \circ \vartheta = S_\vartheta f
		\end{align*}
	in the norm of $\mathrm{C}(K)$ for every $f \in \mathrm{C}(K)$. But this means $\lim_i S_{\vartheta_i} = S_\vartheta$ with respect to the strong operator topology. 
\end{proof}

This leads to the following additional characterizations of topological dynamical systems with discrete spectrum.
\begin{theorem}\label{chardiscrsp}
	For a topological dynamical system $(K,\varphi)$ the following assertions are equivalent.
		\begin{enumerate}[(a)]
			\item The system $(K,\varphi)$ has discrete spectrum.
			\item The uniform enveloping semigroup $\mathrm{E}_{\mathrm{u}}(K,\varphi)$ is a compact topological group.
			\item The system $(K,\varphi)$ is pseudoisometric.
			\item The system $(K,\varphi)$ is equicontinuous.
		\end{enumerate}
	If the space $K$ is metrizable, then these assertions are also equivalent to
		\begin{enumerate}[(c')]
			\item The system $(K,\varphi)$ is isometric.
		\end{enumerate}
\end{theorem}
\begin{proof}
	Assume that (a) holds. By \cref{chardisc} we obtain that  $\mathcal{J}(T^\varphi)$ is a compact topological group and therefore also $\mathrm{E}_{\mathrm{u}}(K,\varphi)$ is a compact topological group by \cref{semigroups}. This shows (b). 
	
	Now assume (b) and take a family $Q$ of pesudometrics generating the topology of $K$. By setting 
		\begin{align*}
			p_q(x,y) \coloneqq \max \{q(\vartheta(x),\vartheta(y))\mid \vartheta \in \mathrm{E}_{\mathrm{u}}(K,\varphi)\}
		\end{align*}
		for $(x,y) \in K$ and $q \in Q$, we obtain a set $P = \{p_q\mid q \in Q\}$ of pseudometrics satisfying the conditions of \cref{structureddef} (ii). Therefore $(K,\varphi)$ is pseudoisometric, and, if $K$ is metrizable, then even isometric. This shows the implications \enquote{(b) $\Rightarrow$ (c)} and \enquote{(b) $\Rightarrow$ (c')}, respectively.
		
		If $(K,\varphi)$ is pseudoisometric, then it is also equicontinuous, showing that (c) (and thus also (c')) implies (d). Finally, assume that (d) holds, i.e., $(K,\varphi)$ is equicontinuous. Then the orbits
			\begin{align*}
				\{f \circ \varphi_t \mid t \in G\}
			\end{align*}
		are equicontinuous and hence relatively compact by the Arzela-Ascoli theorem for every $f \in \mathrm{C}(K)$. Thus, $(K,\varphi)$ has discrete spectrum by \cref{chardisc}.
\end{proof}

We also obtain the following characterization of minimality for systems with discrete spectrum. Here, in analogy to the measure-preserving case, a topological dynamical system $(K,\varphi)$ is called \emph{topologically ergodic} if the \emph{fixed space} of the Koopman representation
	\begin{align*}
		\fix(T^\varphi) \coloneqq \{f \in \mathrm{C}(K) \mid T_{t}^\varphi f = f \textrm{ for every } t \in G\}
	\end{align*}
is one-dimensional (cf. \cite[Section 7]{EdKr2020}).

\begin{proposition}\label{charirred}
	For a topological dynamical system $(K,\varphi)$ with discrete spectrum the following assertions are equivalent.
		\begin{enumerate}[(a)]
			\item The system $(K,\varphi)$ is topologically ergodic.
			\item The system $(K,\varphi)$ is minimal.
			\item The system $(K,\varphi)$ is strictly ergodic, i.e., it is minimal and admits a unique invariant probability measure.
		\end{enumerate}
\end{proposition}
\begin{proof}
	By \cref{chardiscrsp} we obtain that $\mathrm{E}_{\mathrm{u}}(K,\varphi)$ is a compact group. We denote its Haar measure by $m$. As the orbits of the action of $\mathrm{E}_{\mathrm{u}}(K,\varphi)$ are precisely the minimal subsets, the system $(K,\varphi)$ is a disjoint union of its minimal subsets 
	 (this holds even more generally for distal systems, see \cite[Corollary 5.4]{Ausl1988}). 
	
	Now assume that $(K,\varphi)$ contains two minimal subsets $M_1$ and $M_2$. We then find an $f \in \mathrm{C}(K)$ taking the value $0$ on $M_1$ and the value $1$ on $M_2$. But then the function
		\begin{align*}
			\int_{\mathrm{E}_{\mathrm{u}}(K,\varphi)} f \circ \vartheta \, m(\vartheta) \in \mathrm{C}(K)
		\end{align*}
	has the same property and is contained in $\fix(T^\varphi)$. This shows that (a) implies (b). 
	
	If (b) holds, then it is easy to check that for $x \in K$ setting
		\begin{align*}
			\mu(f) \coloneqq \int_{\mathrm{E}_{\mathrm{u}}(K,\varphi)} f(\vartheta(x))\, \mathrm{d}m(\vartheta)
		\end{align*}
	for $f \in \mathrm{C}(K)$ defines an invariant probability measure on $K$. By minimality of the system $(K,\varphi)$, the group $\mathrm{E}_{\mathrm{u}}(K,\varphi)$ acts transitively on $K$ which yields that $\mu$ does not depend on the choice of $x\in K$. If $\nu$ is any other invariant probability measure on $K$, then it is also invariant with respect to the action of $\mathrm{E}_{\mathrm{u}}(K,\varphi)$ and therefore
		\begin{align*}
			\int_K f(x)\, \mathrm{d}\nu(x) = \int_{\mathrm{E}_{\mathrm{u}}(K,\varphi)} \int_K f(\vartheta(x)) \,\mathrm{d}\nu(x)\,\mathrm{d}m(\vartheta) = \langle \langle f,\mu\rangle \cdot \mathbbm{1},\nu \rangle = \langle f , \mu \rangle
		\end{align*}
	by Fubini's theorem.
	
	The implications \enquote{(c) $\Rightarrow$ (b)} and \enquote{(b) $\Rightarrow$ (a)} are obvious.
\end{proof}
\subsection{Measure-preserving systems with discrete spectrum}\label{sdiscspectrumm}
Similarly to the topological case, we now introduce measure-preserving dynamical systems with discrete spectrum.

\begin{definition}
 A measure-preserving system $(\mathrm{X},T)$ \emph{has discrete spectrum} if the underlying strongly continuous group representation $T \colon G \rightarrow \mathscr{L}(\mathrm{L}^1(\uX))$ has discrete spectrum.
\end{definition}

The following are the measure-thereotic versions of the standard examples discussed in \cref{topexample}.

\begin{example}\label{measrot}
	Let $(H,c)$ be a group compactification of $G$ and $(H,\varphi_c)$ the induced group rotation of \cref{topexample}. The Haar measure $m_H$ on $H$ then induces an ergodic measure-preserving system $(H,m_H,T^{\varphi_{c}})$ with discrete spectrum. We call such systems \emph{(measure-preserving) rotations}. 
	
	More generally, given any closed subgroup $U$ of $H$ we equip the quasi-rotation $(H/U,\varphi)$ with the pushforward $m_{H/U}$ of the Haar measure $m$ to again arrive at an ergodic measure-preserving system $(H/U,m_{H/U},T^{\varphi_c})$ with discrete spectrum. Such systems are \emph{(measure-preserving) quasi-rotations}.
\end{example}

We recall the following interesting result from \cite{NaWo1972} related to ergodic systems with discrete spectrum.

\begin{proposition}\label{subinv}
	Let $(\mathrm{X},T)$ be an ergodic measure-preserving system. If $M \subset \mathrm{L}^1(\uX)$ is a finite-dimensional invariant subspace, then $M \subset \mathrm{L}^\infty(\uX)$.
\end{proposition}
\section{Some category theoretical language}\label{scattheory}
To compare the topological and the measure-preserving concepts of discrete spectrum, and to make the statement of the Halmos-von Neumann theorem as precise and concise as possible, we will use some category theoretical language. We briefly recall the necessary concepts, but refer to \cite{MacL1998} or \cite[Appendix 3]{HoMo2006} for a proper introduction. 
\begin{definition}
A \emph{(locally small) category} $\mathscr{C}$ consists of a class $\mathrm{Ob}(\mathscr{C})$ of \emph{objects}, a set $\mathrm{Mor}(X,Y) = \mathrm{Mor}_{\mathscr{C}}(X,Y)$ of \emph{morphisms} for every pair $(X,Y)$ of objects, a \emph{composition map}
	\begin{align*}
		\circ = \circ_{X,Y,Z} \colon \mathrm{Mor}(X,Y) \times \mathrm{Mor}(Y,Z) \rightarrow \mathrm{Mor}(X,Z), \quad (\alpha, \beta) \mapsto \beta \circ \alpha 
	\end{align*}
	for every triple $(X,Y,Z)$ of objects and an identity morphism $\mathrm{id}_X \in \mathrm{Mor}(X,X)$ for every object $X$ satisfying the following conditions.
		\begin{enumerate}[(i)]
			\item For objects $W$, $X$, $Y$, $Z$ we have $\mathrm{Mor}(W,X) \cap \mathrm{Mor}(Y,Z) = \emptyset$ if $W \neq Y$ or $X \neq Z$.
			\item For objects $W$, $X$, $Y$, $Z$ the identity 
				\begin{align*}
					\gamma \circ (\beta \circ \alpha) = (\gamma \circ \beta) \circ \alpha
				\end{align*}
				holds for all $\alpha \in \mathrm{Mor}(W,X)$, $\beta \in \mathrm{Mor}(X,Y)$ and $\gamma \in \mathrm{Mor}(Y,Z)$.
			\item For objects $W$, $X$, $Y$ the identities
				\begin{align*} 
					\beta \circ \mathrm{id}_X = \beta \textrm{ and } \mathrm{id}_{X} \circ \alpha = \alpha
				\end{align*}
				hold for all $\alpha \in \mathrm{Mor}(W,X)$ and $\beta \in \mathrm{Mor}(X,Y)$.
		\end{enumerate}
\end{definition}
Simple examples are the categories of sets (with maps as morphisms) or the category of vector spaces (with linear maps as morphisms). Topological dynamical systems (with respect to the fixed topological group $G$) together with their morphisms  define a category $\mathbf{TopDyn}(G)$. Likewise, measure-preserving systems with their morphisms form a category $\mathbf{MeasDyn}(G)$. 

By taking the Koopman representation of a topological dynamical system we also obtain an object of the following functional analytic category.
\begin{example}\label{cstarcat}
	We call a pair $(A,T)$ a \emph{commutative $\mathrm{C}^*$-dynamical system} if $A$ is a unital commutative $\mathrm{C}^*$-algebra and $T \colon G \rightarrow \mathscr{L}(A)$ is a strongly continuous group representation of $G$ as $^*$-automorphisms. Every topological dynamical system $(K,\varphi)$ induces such a commutative $\mathrm{C}^*$-dynamical system $(\mathrm{C}(K),T^\varphi)$. A \emph{morphism} $J \colon (A,T) \rightarrow (B,S)$ between commutative $\mathrm{C}^*$-dynamical system is a unital $^*$-homorphism $V \colon A \rightarrow B$ such that the diagram
		\[
			\xymatrix{
					B \ar[r]^{S_t}  & B \\
					A\ar[u]^{J}  \ar[r]^{T_t} & A\ar[u]_{J}
				}
		\]	
	commutes for every $t \in G$. For example, if $q \colon (K,\varphi) \rightarrow (L,\psi)$ is a morphism of topological dynamical systems, then the induced operator $J_q \colon \mathrm{C}(L) \rightarrow \mathrm{C}(K), \, f \mapsto f \circ q$ is a morphism between the induced commutative $\mathrm{C}^*$-dynamical systems $(\mathrm{C}(L),T^\psi)$ and $(\mathrm{C}(K),T^\varphi)$. We write $\mathbf{ComC^*Dyn}(G)$ for the category of commutative $\mathrm{C}^*$-dynamical systems.
\end{example}

\begin{remark}
Given any category $\mathscr{C}$ we can perform standard constructions. 
	\begin{enumerate}[(i)]
		\item A \emph{subcategory} $\mathscr{D}$ of $\mathscr{C}$ consists of a subclass of objects of $\mathscr{C}$ and a subset $\mathrm{Mor}_{\mathscr{D}}(X,Y) \subset \mathrm{Mor}_{\mathscr{C}}(X,Y)$ of morphisms for every pair $(X,Y)$ of objects in $\mathscr{D}$ such that 
			\begin{itemize}
				\item $\mathrm{id}_X \in \mathrm{Mor}_{\mathscr{D}}(X,X)$ for every object $X$ in $\mathscr{D}$, and 
				\item $\beta \circ \alpha \in \mathrm{Mor}_{\mathscr{D}}(X,Z)$ for all $\alpha \in \mathrm{Mor}_{\mathscr{D}}(X,Y)$ and $\beta \in \mathrm{Mor}_{\mathscr{D}}(Y,Z)$ where $X,Y,Z$ are objects in $\mathscr{D}$.
			\end{itemize}
		In this case, $\mathscr{D}$ becomes a category in its own right in the obvious way. If $\mathrm{Mor}_{\mathscr{D}}(X,Y) = \mathrm{Mor}_{\mathscr{C}}(X,Y)$ for all objects $X$ and $Y$ in $\mathscr{D}$, then $\mathscr{D}$ is called a \emph{full subcategory of} $\mathscr{C}$.
		\item The \emph{opposite category} $\mathscr{C}^{\mathrm{op}}$ has the same objects $\mathrm{Ob}(\mathscr{C}^{\mathrm{op}}) \coloneqq \mathrm{Ob}(\mathscr{C})$ but the reversed morphisms: $\mathrm{Mor}_{\mathscr{C}^{\mathrm{op}}}(X,Y) \coloneqq \mathrm{Mor}_{\mathscr{C}}(Y,X)$ for all objects $X$ and $Y$. The composition maps are then also reversed.
	\end{enumerate}
\end{remark}
	
To compare categories, we also need the concept of functors.

\begin{definition}
Given categories $\mathscr{C}$ and $\mathscr{D}$, a \emph{functor} $F\colon \mathscr{C} \rightarrow \mathscr{D}$ consists of
			\begin{enumerate}[(i)]
				\item  an assignment $F \colon \mathrm{Ob}(\mathscr{C}) \rightarrow \mathrm{Ob}(\mathscr{D})$ and, 
				\item a map
			\begin{align*}
				F =  F_{X,Y} \colon \mathrm{Mor}(X,Y) \rightarrow \mathrm{Mor}(X,Y), \quad \alpha \mapsto F(\alpha)
			\end{align*}
		for every pair $(X,Y)$ of objects in $\mathscr{C}$,
			\end{enumerate}
		such that 
			\begin{enumerate}[(1)]
				\item $F(\beta \circ \alpha ) = F(\beta) \circ F(\alpha)$ for all $\alpha \in \mathrm{Mor}(X,Y)$ and $\beta \in \mathrm{Mor}(Y,Z)$ where $X,Y,Z$ are objects in $\mathscr{C}$, and
				\item $F(\mathrm{id}_X) = \mathrm{Id}_{F(X)}$ for every object $X$ in $\mathscr{C}$.
			\end{enumerate}
\end{definition}

On any category $\mathscr{C}$ there is the \emph{identity functor} $\mathrm{Id}_{\mathscr{C}}\colon \mathscr{C} \rightarrow \mathscr{C}$ assigning every object and every morphism to itself. Given two functors $F_1 \colon \mathscr{C} \rightarrow \mathscr{D}$ and $F_2 \colon \mathscr{D} \rightarrow \mathscr{E}$ we can also consider the \emph{composition} $F_2 \circ F_1$ which is defined in the obvious way. 

\begin{example}\label{koopmanfunct}
		For a topological dynamical system $(K,\varphi)$ we set $\mathrm{Koop}(K,\varphi) \coloneqq (\mathrm{C}(K),T^\varphi)$. If $q \colon (K,\varphi) \rightarrow (L,\psi)$ is a morphism between topological dynamical systems we take $\mathrm{Koop}(q)$ as the induced morphism $J_q$ from $(\mathrm{C}(L),T^\psi)$ to $(\mathrm{C}(K),T^\varphi)$ (cf. \cref{cstarcat}). Then $\mathrm{Koop}$ is a functor from the category $\mathbf{TopDyn}(G)$ of topological dynamical systems to the opposite category $\mathbf{ComC^*Dyn}(G)^{\mathrm{op}}$ of commutative $\mathrm{C}^*$-dynamical systems (since the direction of morphisms is reversed). We call this the \emph{Koopman functor}.
	\end{example}
	By Gelfand's representation theory for commutative $\mathrm{C}^*$-algebras we also obtain a functor in the other direction.
	\begin{example}\label{gelfand}
	 	Let $G$ be a topological group and $(A,T)$ be a commutative $\mathrm{C}^*$-dynamical system. Then the \emph{spectrum} (or \emph{Gelfand space})
	 		\begin{align*}
	 			\mathrm{Spec}(A) \coloneqq \{\chi \in A'\mid 0 \neq \chi \textrm{ unital algebra  homomorphism}\} \subset A'
	 		\end{align*}
	 	equipped with the weak* topology is a compact space. By restricting the adjoints, i.e., setting $\varphi_t \coloneqq T_{-t}'|_{\mathrm{Spec}(A)}$ for $t \in G$, we obtain a topological dynamical system $\mathrm{Spec}(A,T) \coloneqq (\mathrm{Spec}(A),\varphi)$. Moreover, every morphism $J \colon (A,T) \rightarrow (B,S)$ between commutative $\mathrm{C}^*$-dynamical systems defines a morphism $q \coloneqq J'|_{\mathrm{Spec}(B)} \colon \mathrm{Spec}(B,S) \rightarrow \mathrm{Spec}(A,T)$. We therefore obtain a functor $\mathrm{Spec}$ from the category  $\mathbf{ComC^*Dyn}(G)^{\mathrm{op}}$ to the category $\mathbf{TopDyn}(G)$ which we call the \emph{Gelfand functor}.
	\end{example} 
Finally, we recall the concepts of natural transformations between functors and equivalence of two categories.
	\begin{definition}\label{natiso}
		Given categories $\mathscr{C}$ and $\mathscr{D}$ and functors $F_1, F_2 \colon \mathscr{C} \rightarrow \mathscr{D}$, a \emph{natural transformation} $\eta \colon F_1 \rightarrow F_2$ assigns to every object $X$ of $\mathscr{C}$ a morphism $\eta_X \in \mathrm{Mor}_{\mathscr{D}}(F_1(X), F_2(X))$ such that the diagram commutes
			\[
				\xymatrix{
					F_1(X) \ar[r]^{\eta_X} \ar[d]_{F_1(\alpha)} & F_2(X) \ar[d]^{F_2(\alpha)}\\
					F_1(Y) \ar[r]^{\eta_Y} & F_2(Y)
				}
			\]
		for all morphisms $\alpha \in \mathrm{Mor}(X,Y)$ and objects $X$, $Y$ of $\mathscr{C}$.
		If $\eta_X$ is an isomorphism for every object $X$, then $\eta$ is called a \emph{natural isomorphism}.
		
		Two functors $F_1\colon \mathscr{C} \rightarrow \mathscr{D}$ and $F_2 \colon \mathscr{D} \rightarrow \mathscr{C}$ are \emph{essentially inverse to each other} if  $F_2 \circ F_1$ is naturally isomorphic to $\mathrm{Id}_{\mathscr{C}}$, and $F_1 \circ F_2$ is naturally isomorphic to $\mathrm{Id}_{\mathscr{D}}$. In this case, the categories $\mathscr{C}$ and $\mathscr{D}$ are \emph{equivalent}.
	\end{definition}

\begin{example}\label{pointeval}
	If $(K,\varphi)$ is a topological dynamical system, we obtain an isomorphism 
		\begin{align*}
			\gamma_{(K,\varphi)} \colon (K,\varphi) \rightarrow \mathrm{Spec}(\mathrm{Koop}(K,\varphi))
		\end{align*}
	of topological dynamical systems by mapping each point $x \in K$ to the corresponding point evaluation $\delta_x \colon \mathrm{C}(K) \rightarrow \C, \, f \mapsto f(x)$ (see \cite[Theorem 4.11]{EFHN2015}). One readily checks that $\gamma \colon \mathrm{Id}_{\mathbf{TopDyn}(G)} \rightarrow \mathrm{Spec} \circ \mathrm{Koop}$ is a natural isomorphism.
\end{example}

\begin{example}\label{gelfandmap}
	If $(A,T)$ is a commutative $\mathrm{C}^*$-dynamical system, the Gelfand map $A \rightarrow \mathrm{C}(\mathrm{Spec}(A)), \, x \mapsto [\chi \mapsto \chi(x)]$ defines an isomorphism 
		\begin{align*}
			\gamma_{(A,T)} \colon (A,T) \rightarrow \mathrm{Koop}(\mathrm{Spec}(A,T)),
		\end{align*}
	 see, e.g., \cite[Section 4.4 and the Supplement of Chapter 4]{EFHN2015}. A moment's thought reveals that $\gamma \colon \mathrm{Id}_{\mathbf{ComC^*Dyn}(G)^{\mathrm{op}}} \rightarrow \mathrm{Koop} \circ \mathrm{Spec}$ is actually a natural isomorphism.
\end{example}
As a consequence of the previous two examples, we immediately obtain the following well-known but fundamental consequence of the Gelfand-Naimark representation theorem for commutative $\mathrm{C}^*$-algebras.
\begin{theorem}\label{gelfandcat}
	The functors 
		\begin{align*}
			\mathrm{Koop} &\colon \mathbf{TopDyn}(G) \rightarrow \mathbf{ComC^*Dyn}^{\mathrm{op}}(G), \\
			\mathrm{Spec}& \colon \mathbf{ComC^*Dyn}^{\mathrm{op}}(G) \rightarrow \mathbf{TopDyn}(G)
		\end{align*}
	establish an equivalence between the category of topological dynamical systems and the opposite category of commutative $\mathrm{C}^*$-dynamical systems.
\end{theorem}

\section{Equivalence between topological and measure-preserving systems with discrete spectrum}\label{sequivalence}
We now show that, from a category theoretic point of view, the worlds of minimal topological dynamical systems with discrete spectrum on one hand, and ergodic measure-preserving systems with discrete spetrum on the other, are equivalent. This allows to prove a topological Halmos-von Neumann theorem and then derive the ergodic theory version simply by using this correspondence.

In the following we write $\mathbf{MinDisc}(G)$ for the full subcategory of $\mathbf{TopDyn}(G)$ having as objects the minimal systems with discrete spectrum. Likewise, the full subcategory of $\mathbf{MeasDyn}(G)$ of ergodic systems with discrete spectrum as its objects is denoted by $\mathbf{ErgDisc}(G)$.

\begin{construction}\label{measfunctor}
	Define a functor $\mathrm{Meas} \colon \mathbf{MinDisc}(G) \rightarrow \mathbf{ErgDisc}(G)^{\mathrm{op}}$ by setting
		\begin{enumerate}[(i)]
			\item $\mathrm{Meas}(K,\varphi) \coloneqq (K,\mu,T^\varphi)$ for every minimal topological dynamical system $(K,\varphi)$ with discrete spectrum, where $\mu$ is the unique invariant probability measure (see \cref{charirred}), and
			\item $\mathrm{Meas}(q) \coloneqq J_q$ for a morphism $q \colon (K,\varphi) \rightarrow (L,\psi)$ between such systems, where $J_q$ is the unique continuous extension of the induced operator $J_q \in \mathscr{L}(\mathrm{C}(L),\mathrm{C}(K))$ to the corresponding $\mathrm{L}^1$-spaces.
		\end{enumerate}
\end{construction}

To construct a functor in the other direction we first prove the following easy lemma.

\begin{lemma}\label{discspsubalgebra}
	Let $(\uX,T)$ be a measure-preserving system. Then
		\begin{align*}
			\mathrm{L}^\infty(\uX)_{\mathrm{disc}} =  \overline{\bigcup \{M \subset \mathrm{L}^1(\uX)\mid M \textrm{ finite-dimensional invariant subspace}\}}^{\mathrm{L}^\infty} \subset \mathrm{L}^\infty(\uX).
		\end{align*}	
	is an invariant unital $\mathrm{C}^*$-subalgebra of $\mathrm{L}^\infty(\uX)$ and the map
		\begin{align*}
			T_{\mathrm{disc}} \colon G \rightarrow \mathscr{L}(\mathrm{L}^\infty(\uX)_{\mathrm{disc}}), \quad t \mapsto T_t|_{\mathrm{L}^\infty(\uX)_{\mathrm{disc}}}
		\end{align*}
	is a bounded strongly continuous group representation as $^*$-automorphisms of $\mathrm{L}^\infty(\uX)_{\mathrm{disc}}$.
\end{lemma}
\begin{proof}
	For every $t \in G$ the operator $T_t \in \mathscr{L}(\mathrm{L}^1(\uX))$ restricts to a $^*$-automorphism $T_t|_{\mathrm{L}^\infty(\uX)} \in \mathscr{L}(\mathrm{L}^\infty(\uX))$ (see \cite[Theorem 13.9]{EFHN2015}). The only remaining non-trivial assertion is strong continuity of the restricted representation. However, since the topologies induced by the $\mathrm{L}^\infty$ and $\mathrm{L}^1$-norms agree on finite dimensional subspaces, we obtain that there is a dense subset $D \subset \mathrm{L}^\infty(\uX)_{\mathrm{disc}}$ such that
		\begin{align*}
			G \rightarrow \mathrm{L}^\infty(\uX)_{\mathrm{disc}}, \quad t \mapsto T_tf
		\end{align*}
	is continuous for every $f \in D$. Since the representation is bounded, this already shows strong continuity (use \cite[Proposition C.18]{EFHN2015}).
\end{proof}

In other words, the pair $(\mathrm{L}^\infty(\uX)_{\mathrm{disc}},T_{\mathrm{disc}})$ is a commutative $\mathrm{C}^*$-dynamical system (cf. \cref{cstarcat}). Applying \cref{gelfandcat}, we arrive at a topological dynamical system with discrete spectrum. We use this observation to perform the following construction (keeping the notation of \cref{discspsubalgebra}).

\begin{construction}\label{topfunctor}
	We define a functor $\mathrm{Top} \colon \mathbf{ErgDisc}(G)^{\mathrm{op}} \rightarrow \mathbf{MinDisc}(G)$ by setting
		\begin{enumerate}[(i)]
			\item $\mathrm{Top}(\uX,T) \coloneqq \mathrm{Spec}(\mathrm{L}^\infty(\uX)_{\mathrm{disc}},T_{\mathrm{disc}})$  for every ergodic measure-preserving system $(\uX,T)$ with discrete spectrum, and 
			\item $\mathrm{Top}(J)\coloneqq \mathrm{Spec}(J|_{\mathrm{L}^\infty(\uY)_{\mathrm{disc}}})$ for every morphism $J \colon (\uY,S) \rightarrow (\uX,T)$ between such systems.
		\end{enumerate}

\end{construction}
Note here that, for an ergodic measure-preserving system $(\uX,T)$ with discrete spectrum, the constructed topological dynamical system $\mathrm{Spec}(\mathrm{L}^\infty(\uX)_{\mathrm{disc}},T_{\mathrm{disc}})$ with discrete spectrum is in fact minimal by \cref{charirred}.

To show that the functors $\mathrm{Meas}$ of \cref{measfunctor} and $\mathrm{Top}$ of \cref{topfunctor} are essentially inverse to each other, we now have to prove that the compositions $\mathrm{Top} \circ \mathrm{Meas}$ and $\mathrm{Meas} \circ \mathrm{Top}$ are naturally isomorphic to the identity functors on $\mathbf{MinDisc}(G)$ and $\mathbf{ErgDisc}(G)$, respectively. For the first composition we need the following lemma.

\begin{lemma}\label{comingback}
	Let $(K,\varphi)$ be a minimal system with discrete spectrum and $\mu$ its unique invariant probability measure. If $M \subset \mathrm{L}^\infty(K,\mu)$ is a finite-dimensional, invariant subspace, then $M \subset \mathrm{C}(K)$.
\end{lemma}
\begin{proof}
	Observe that $H \defeq \mathrm{E}_{\mathrm{u}}(K,\varphi)$ is a compact group (see \cref{chardiscrsp}) acting continuously and transitively on $K$. This allows us to apply a Fourier analytic result from \cite[Section 6]{EdKr2020}: For every equivalence class $[\pi] \in \hat{H}$ of irreducible finite-dimensional representations of $H$ we obtain an operator $P_{[\pi]} \in \mathscr{L}(\mathrm{L}^2(K,\mu))$ via
		\begin{align*}
			P_{[\pi]}f \coloneqq \mathrm{dim}([\pi]) \int_H \mathrm{tr}([\pi])(\vartheta) \cdot S_{\vartheta^{-1}}f \, \mathrm{d}m_H(\vartheta) \textrm{ for } f \in \mathrm{L}^2(K,\mu),
		\end{align*}
	where $\mathrm{dim}([\pi])$ and $\mathrm{tr}([\pi])$ denote the dimension and trace of $[\pi]$, respectively, $S_{\vartheta} f \coloneqq f \circ \vartheta$ for $\vartheta \in H$, and the integral is understood in the weak sense (see \cite[Definition 3.26 and Theorem 3.27]{Rudi1991}). Moreover,  $P_{[\pi]}f$ is contained in the closed linear hull of the orbit
		\begin{align*}
			\overline{\mathrm{lin}}\{S_\vartheta f \mid \vartheta \in H\}
		\end{align*}
	for $f \in \mathrm{L}^2(K,\mu)$ (see again \cite[Theorem 3.27]{Rudi1991}), and this implies $P_{[\pi]}M \subset M$ by invariance of $M$ with respect to $T^\varphi$ and continuity of the map $H \rightarrow \mathrm{L}^2(K,\mu), \, \vartheta \mapsto S_{\vartheta}f$. By  \cite[Theorem 6.1]{EdKr2020} the operators $P_{[\pi]}$ for $[\pi] \in \hat{H}$ are pairwise orthogonal projections with ranges in $\mathrm{C}(K)$ and
			\begin{align*}
				f = \sum_{[\pi] \in \hat{H}} P_{[\pi]}f
			\end{align*}
		for every $f \in \mathrm{L}^2(K,\mu)$.\footnote{In the cited article the operators $P_{[\pi]}$ are defined via pointwise integrals, but one can readily check that the two definitions agree, cf. the end of \cite[Appendix 3]{Foll2016}.}
		Now if $f \in M$, then this series representation is actually a finite sum since $M$ is finite-dimensional, and we conclude that $f \in \mathrm{C}(K)$.
\end{proof}
Thus, if we start from a minimal topological dynamical system $(K,\varphi)$ with discrete spectrum, equip it with its unique invariant probability measure $\mu$ and then perform the construction of \cref{topfunctor}, we come back to the original system up to a canonical isomorphism. More precisely, we obtain the following.
\begin{proposition}\label{natiso1}
	The functor $\mathrm{Top} \circ \mathrm{Meas}\colon \mathbf{MinDisc}(G) \rightarrow\mathbf{MinDisc}(G)$	is the restriction of the functor $\mathrm{Spec} \circ \mathrm{Koop}\colon \mathbf{TopDyn}(G) \rightarrow \mathbf{TopDyn}(G)$ to the full subcategory $\mathbf{MinDisc}(G)$. In particular, the natural isomorphism $\gamma$ from \cref{pointeval} restricts to a natural isomorphism 
		\begin{align*}
			\gamma \colon \mathrm{Id}_{\mathbf{MinDisc}(G)} \rightarrow \mathrm{Top} \circ \mathrm{Meas}.
		\end{align*}	
\end{proposition}

With the help of the following lemma we show that also the composition $\mathrm{Meas} \circ \mathrm{Top}\colon \mathbf{ErgDisc}(G) \rightarrow \mathbf{ErgDisc}(G)$ is naturally isomorphic to the identity functor.
\begin{lemma}\label{prepnatiso2}
	Let $(\uX,T)$ be an ergodic measure-preserving system with discrete spectrum. Then the Gelfand morphism $(\mathrm{L}^\infty(\uX)_{\mathrm{disc}},T_{\mathrm{disc}}) \rightarrow \mathrm{Koop}(\mathrm{Spec}(\mathrm{L}^\infty(\uX)_{\mathrm{disc}},T_{\mathrm{disc}}))$ from \cref{gelfandmap} extends uniquely to an isomorphism 
		\begin{align*}
			\gamma_{(\uX,T)} \colon (\uX,T) \rightarrow\mathrm{Meas}(\mathrm{Spec}(\mathrm{L}^\infty(\uX)_{\mathrm{disc}},T_{\mathrm{disc}}))).
		\end{align*}
\end{lemma}
\begin{proof}
	We write $(K,\varphi)$ for the minimal system $\mathrm{Spec}(\mathrm{L}^\infty(\uX)_{\mathrm{disc}},T_{\mathrm{disc}}) = \mathrm{Top}(\uX,T)$ with discrete spectrum constructed in \cref{topfunctor} and $\mu$ for its unique invariant probability measure (see \cref{charirred}). Moreover, let $V \colon (\mathrm{L}^\infty(\uX)_{\mathrm{disc}},T_{\mathrm{disc}}) \rightarrow (\mathrm{C}(K),T^\varphi)$ be the Gelfand isomorphism from \cref{gelfandmap}. By transporting the integration functional 
	\begin{align*}
		\langle \cdot, \mathbbm{1} \rangle \colon \mathrm{L}^\infty(\uX)_{\mathrm{disc}} \rightarrow \C, \quad f \mapsto \int f \, \mathrm{d}\mu_{\mathrm{X}}
	\end{align*}
	we obtain a functional $(V^{-1})'\langle \cdot, \mathbbm{1} \rangle$ on $\mathrm{C}(K)$. Using the Riesz representation theorem we identify this with an invariant probability measure on $K$, hence $(V^{-1})'\langle \cdot, \mathbbm{1} \rangle = \mu$. Since $|Vf| = V|f|$ for every $f \in \mathrm{C}(K)$ (see \cite[Theorem 7.23]{EFHN2015}), one now readily checks that $V$ is an isometry with respect to the corresponding $\mathrm{L}^1$-norms and intertwines the representations (see \cite[Section 12.3]{EFHN2015}). Since $V$ also has dense range by \cref{subinv}, we obtain that $V$ has a unique extension to a Markov lattice isomorphism $\mathrm{L}^1(K,\mu) \rightarrow \mathrm{L}^1(\uX)$ which then defines the desired isomorphism $\gamma_{(\uX,T)}$. 
\end{proof}

With the notation of \cref{prepnatiso2} we now obtain the counterpart to \cref{natiso1}.
\begin{proposition}\label{natiso2}
	The isomorphisms $\gamma_{(\uX,T)}$ for ergodic measure-preserving sysystems $(\uX,T)$ with discrete spectrum from \cref{prepnatiso2} define a natural isomorphism 
		\begin{align*}
			\gamma \colon \mathrm{Id}_{\mathbf{ErgDisc}(G)} \rightarrow \mathrm{Meas} \circ \mathrm{Top}.
		\end{align*}
\end{proposition}
\begin{proof}
	Using the definitions of the functors $\mathrm{Top}$ and $\mathrm{Meas}$ as well as the definition of the isomorphisms $\gamma_{(\uX,T)}$, it is straightforward to check that the diagram in \cref{natiso} actually commutes. We omit the details.
\end{proof}
Combining \cref{natiso1} and \cref{natiso2} we now finally arrive at the desired equivalence.

\begin{theorem}\label{equivtopmeas}
	The functors 
		\begin{align*}
			\mathrm{Meas}&\colon \mathbf{MinDisc}(G) \rightarrow \mathbf{ErgDisc}(G)^{\mathrm{op}},\\
			\mathrm{Top}&\colon  \mathbf{ErgDisc}(G)^{\mathrm{op}} \rightarrow \mathbf{MinDisc}(G)
		\end{align*}
	establish an equivalence between the category of minimal topological dynamical systems with discrete spectrum and the opposite category of ergodic measure-preserving systems with discrete spectrum.
\end{theorem}
As a simple consequence we obtain the following.
\begin{corollary}
	For two minimal topological dynamical systems with discrete spectrum  $(K_1,\varphi_1)$ and $(K_2,\varphi_2)$ with unique invariant probability measures $\mu_1$ and $\mu_2$, respectively, the following assertions are equivalent.
		\begin{enumerate}[(a)]
			\item The topological dynamical systems $(K_1,\varphi_1)$ and $(K_2,\varphi_2)$ are isomorphic.
			\item The induced measure-preserving systems $(K_1,\mu_1,T^{\varphi_1})$ and $(K_2,\mu_2,T^{\varphi_2})$ are isomorphic.
		\end{enumerate}		
\end{corollary}

\section{Halmos-von Neumann for abelian groups}
\subsection{The classical Halmos-von Neumann theorem for abelian groups}\label{sclassicalcom}
We now state the Halmos-von Neumann for actions of an abelian group $G$ which will be fixed for the whole section. A treatment of the result for measure-preserving systems is given in \cite[Subsection 1.5]{JaShTa2021} and \cite[Theorem 2.6]{Stor1974}. Here, we use a topological approach and start with the following easy generalization of the point spectrum of operators.
	\begin{definition}\label{defpointspect}
		Let $T \colon G \rightarrow \mathscr{L}(E), \, t \mapsto T_t$ be a bounded strongly continuous group representation on a Banach space $E$. We say that $\chi \in G^*$ is an \emph{eigencharacter} of $T$ if there is $x \in E\setminus \{0\}$ with $T_tx = \chi(t)x$ for every $t \in G$. We call the set
			\begin{align*}
				\sigma_{\mathrm{p}}(T) \coloneqq \{\chi \in G^*\mid \chi \textrm{ eigencharacter of } T\} \subset G^*
			\end{align*}
		the \emph{point spectrum of $T$}.
		
		For  a topological dynamical system $(K,\varphi)$ we write $\sigma_{\mathrm{p}}(K,\varphi) \coloneqq \sigma_{\mathrm{p}}(T^\varphi)$, and for a measure-preserving system $(\uX,T)$ we set $\sigma_{\mathrm{p}}(\uX,T) \coloneqq \sigma_{\mathrm{p}}(T)$.
	\end{definition}
	
\begin{remark}
	If $G = \Z$, the point spectrum of a bounded strongly continuous group representation $T \colon G \rightarrow \mathscr{L}(E), \, t \mapsto T_t$ can be identified with the point spectrum $\sigma_{\mathrm{p}}(T_1)$ via the canonical isomorphism $\Z^* \cong \T$. 
\end{remark}

We make the following important observation.
	\begin{proposition}\label{subgroup1}
		\begin{enumerate}[(i)]
			\item Let $(K,\varphi)$ be a topologically ergodic dynamical system. Then $\sigma_{\mathrm{p}}(K,\varphi)$ is a subgroup of $G^*$.
			\item Let $(\uX,T)$ be an ergodic measure-preserving system. Then $\sigma_{\mathrm{p}}(\uX,T)$ is a subgroup of $G^*$.
		\end{enumerate}
	\end{proposition}
	\begin{proof}
		Consider (i). Clearly, $T_t^\varphi \mathbbm{1} = \mathbbm{1}$ for every $t \in G$ and hence $\sigma_{\mathrm{p}}(K,\varphi)$ contains the neutral element of $G^*$. Now if $f \in \mathrm{C}(K)$ and $\chi \in G^*$ with $T_t^\varphi f = \chi(t) f$ for all $t \in G$, then $T_t^\varphi|f| = |T_t^\varphi f| = |f|$ for every $t \in G$, i.e., $|f| \in \fix(T) = \C \mathbbm{1}$.  We thus obtain that $|f|$ is constant. This implies that, whenever $f,g \in \mathrm{C}(K)\setminus\{0\}$ and $\chi,\tau \in G^*$ with $T_t^\varphi f = \chi(t) f$ and $T_t^\varphi g = \tau(t) g$ for every $t \in G$, $g$ is invertible and $f \cdot g^{-1} \neq 0$. Since $T_t^\varphi (fg^{-1}) = (\chi\cdot\tau^{-1})(t) fg^{-1}$ for $t \in G$, this implies $\chi\cdot\tau^{-1} \in \sigma_{\mathrm{p}}(T^\varphi)$. Therefore $\sigma_{\mathrm{p}}(K,\varphi) = \sigma_{\mathrm{p}}(T^\varphi)$ is a subgroup of $G^*$.
		
		For (ii) note that for every $\chi \in G^*$ every element of the space
			\begin{align*}
				\bigcap_{t \in G} \ker(\chi(t) - T_t)
			\end{align*}
		is in $\mathrm{L}^\infty(\uX)$ by \cref{subinv}. Working in $\mathrm{L}^\infty(\uX)$ instead of $\mathrm{L}^1(\uX)$ we then proceed exactly as in the proof of part (i).
	\end{proof}
	
We now state the topological Halmos-von Neumann theorem for commutative groups.
 
\begin{theorem}\label{comhvn1}
	Assume that $G$ is commutative. For minimal topological dynamical systems with discrete spectrum the following assertions hold.
	\begin{enumerate}[(i)]
    	\item Two minimal systems $(K_1,\varphi_1)$ and $(K_2,\varphi_2)$ with discrete spectrum are isomorphic if and only if $\sigma_{\mathrm{p}}(K_1,\varphi_1) = \sigma_{\mathrm{p}}(K_2,\varphi_2)$. 
    	\item For every subgroup $\sigma \subset G^*$ there is a minimal system $(K,\varphi)$ with discrete spectrum such that $\sigma_{\mathrm{p}}(K,\varphi) = \sigma$.
    	\item For every minimal system $(K,\varphi)$ with discrete spectrum there is a group compactification $(H,c)$ such that $(K,\varphi)$ is isomorphic to the rotation $(H,\varphi_c)$.
	\end{enumerate}
\end{theorem}
The ergodic theoretic counterpart is the following.
\begin{theorem}\label{comhvn2}
	Assume that $G$ is commutative. For ergodic measure-preserving dynamical systems with discrete spectrum the following assertions hold.
	\begin{enumerate}[(i)]
    	\item Two ergodic systems $(\uX_1,T_1)$ and $(\uX_2,T_2)$ with discrete spectrum are isomorphic if and only if $\sigma_{\mathrm{p}}(\uX_1,T_1) = \sigma_{\mathrm{p}}(\uX_2,T_2)$. 
    	\item For every subgroup $\sigma \subset G^*$ there is an ergodic system $(\uX,T)$ with discrete spectrum such that $\sigma_{\mathrm{p}}(\uX,T) = \sigma$.
    	\item For every ergodic system $(\uX,T)$ with discrete spectrum there is a group compactification $(H,c)$ such that $(\uX,T)$ is isomorphic to the rotation $(H,m_H,T^{\varphi_c})$.
	\end{enumerate}
\end{theorem}

Given an ergodic measure-preserving system $(\uX,T)$ with discrete spectrum, we obtain that $\sigma_{\mathrm{p}}(\mathrm{Top}(\uX,T)) = \sigma_{\mathrm{p}}(\uX,T)$. Conversely, it follows from \cref{comingback} that $\sigma_{\mathrm{p}}(\mathrm{Meas}(K,\varphi)) = \sigma_{\mathrm{p}}(K,\varphi)$ for every minimal topological dynamical system $(K,\varphi)$ with discrete spectrum. Applying \cref{equivtopmeas} we conclude that \cref{comhvn1} and \cref{comhvn2} are actually equivalent statements. It therefore suffices to prove one of them. We opt to prove the topological version in the next section, but in a different formulation.

\subsection{A categorical topological Halmos-von Neumann theorem for abelian groups}\label{scatcomm}

Using category theoretical language, the topological Halmos-von Neumann theorem \cref{comhvn1} can be expressed in a more concise and, in some sense, even more precise way.
To cover the repesentation aspect (i.e., part (iii)) of the result we consider the following category.

\begin{definition}\label{catgroupcomp}
	The \emph{category of group compactifications} has as objects group compactifications $(H,c)$ of $G$, and as morphisms $\Phi \colon (H_1,c_1) \rightarrow (H_2,c_2)$ between group compactifications $(H_1,c_1)$ and $(H_2,c_2)$ of $G$ continuous group homomorphisms $\Phi \colon H_1 \rightarrow H_2$ such that $\Phi \circ c_1 = c_2$. We denote this category by $\mathbf{Comp}(G)$.
\end{definition}

\begin{remark}\label{morphsurj}
	Note that every morphism of group compactifications is necessarily surjective.
\end{remark}

In view of \cref{topexample}, every group compactifications gives rise to a minimal system with discrete spectrum. The following lemma allows us to also go the other direction by using the uniform enveloping semigroup (introduced in \cref{uniformenv} above).

	\begin{lemma}\label{prepcompfunct}
		\begin{enumerate}[(i)]
			\item If $(K,\varphi)$ is a minimal system with discrete spectrum, then $\mathrm{E}_{\mathrm{u}}(K,\varphi)$ combined with the map $i = i_{(K,\varphi)} \colon G \rightarrow \mathrm{E}_{\mathrm{u}}(K,\varphi),\, t \mapsto \varphi_t$ is a group compactification of $G$.
			\item For a morphism $q \colon (K,\varphi) \rightarrow (L,\psi)$ between minimal system with discrete spectrum, the map
				\begin{align*}
					p_q \colon \mathrm{E}_{\mathrm{u}}(K,\varphi) \rightarrow \mathrm{E}_{\mathrm{u}}(L,\psi), \quad \vartheta \mapsto p_q(\vartheta)
				\end{align*}
			with $p_q(\vartheta)(q(x)) \coloneqq q(\vartheta(x))$ for $x \in K$ and $\vartheta \in \mathrm{E}_{\mathrm{u}}(K,\varphi)$ defines a morphism $p_q \colon (\mathrm{E}_{\mathrm{u}}(K,\varphi),i_{(K,\varphi)}) \rightarrow (\mathrm{E}_{\mathrm{u}}(L,\psi),i_{(L,\psi)})$ of group compactifications.
		\end{enumerate}
	\end{lemma}
	\begin{proof}
		Part (i) is obvious. For part (ii), observe that if $x,y \in K$ satisfy $q(x) = q(y)$, then also 
			\begin{align*}
				q(\varphi_t(x)) = \psi_t(q(x)) = \psi_t(q(y)) = q(\varphi_t(y))
			\end{align*}
		for all $t \in G$, and hence $q(\vartheta(x)) = q(\vartheta(y))$ for every $\vartheta \in \mathrm{E}_{\mathrm{u}}(K,\varphi)$. Thus, $p_q$ is well-defined. It is easily checked that $p_q$ is actually a morphism of group compactifications of $G$.\footnote{Part (ii) is true for extensions of arbitrary systems if one uses the Ellis semigroup instead of the uniform enveloping group and only demands that $p_q$ is a continuous surjective semigroup homomorphism (see \cite[Theorem 3.7]{Ausl1988}). For systems with discrete spectrum the two enveloping semigroups coincide and are actually compact groups by \cref{chardiscrsp}. Moreover, it is evident that $p_q$ is even a morphism of group compactifications in this case.}
	\end{proof}
	
\cref{prepcompfunct} allows us to define a functor $F \colon \mathbf{MinDisc}(G) \rightarrow \mathbf{Comp}(G)$ from the category of minimal systems with discrete spectrum to the category  of group compactifications by setting
	\begin{itemize}
				\item $F(K,\varphi) \coloneqq (\mathrm{E}_{\mathrm{u}}(K,\varphi), i_{(K,\varphi)})$ for every minimal system $(K,\varphi)$ with discrete spectrum, and
				\item $F(q) \coloneqq p_q$ for every morphism $q \colon (K,\varphi) \rightarrow (L,\psi)$ of minimal systems with discrete spectrum.
	\end{itemize}
In general however, this functor does not define a categorical equivalence as the following remark shows.

\begin{remark}\label{problemequiv}
	Assume that $G$ is commutative and that $(K,\varphi)$ is a non-trivial minimal system with discrete spectrum. Pick $\vartheta \in \mathrm{E}_{\mathrm{u}}(K,\varphi) \setminus \{\mathrm{id}_K\}$ and observe that $q \coloneqq \vartheta \colon K \rightarrow K$ defines an automorphism of the system $(K,\varphi)$ (since $\mathrm{E}_{\mathrm{u}}(K,\varphi)$ is abelian) which is not the identity. However, with the notation of \cref{prepcompfunct} we obtain that 
		\begin{align*}
			p_q(\varrho)(q(x)) = \vartheta(q(x)) = \vartheta(\varrho(x)) = \varrho(\vartheta(x)) =  \varrho(q(x))
		\end{align*}
	for every $x \in K$ and $\varrho \in \mathrm{E}_{\mathrm{u}}(K,\varphi)$. We therefore obtain that $p_q(\varrho) = \varrho$ for every $\varrho \in \mathrm{E}_{\mathrm{u}}(K,\varphi)$, showing that $p_q$ is the identity. We conclude that the map
		\begin{align*}
			\mathrm{Mor}((K,\varphi),(K,\varphi)) \rightarrow \mathrm{Mor}(\mathrm{E}_{\mathrm{u}}(K,\varphi),\mathrm{E}_{\mathrm{u}}(K,\varphi)), \quad q \mapsto p_q
		\end{align*}
	is not injective, which implies that the functor $F$ cannot define a categorical equivalence.
\end{remark}

The issue discussed in \cref{problemequiv} arises from the fact that for a group rotation system there is always a distinct point (the neutral element) whereas there is no such point for a general minimal system with discrete spectrum. However, this problem is only minor and can be fixed by choosing a \enquote{base point} for every system.

\begin{definition}
	A \emph{pointed minimal system with discrete spectrum} is a triple $(K,\varphi,y)$ where $(K,\varphi)$ is a minimal system with discrete spectrum and $y \in K$ is a point of $K$. A morphism $q \colon (K,\varphi,y) \rightarrow (L,\psi,z)$ between such systems is a morphism $q \colon (K,\varphi) \rightarrow (L,\psi)$ between the corresponding topological dynamical systems with $q(y) = z$. We write $\mathbf{MinDisc}_{\mathbf{pt}}(G)$ for this category.
\end{definition}

Group compactifications induce pointed minimal systems with discrete spectrum in a natural and functorial way.

\begin{construction}\label{rotfunct}
	We define a functor $\mathrm{Rot}\colon \mathbf{Comp}(G) \rightarrow \mathbf{MinDisc}_{\mathbf{pt}}(G)$ by setting
		\begin{enumerate}[(i)]
			\item $\mathrm{Rot}(H,c) \coloneqq (H,\varphi_c,\mathrm{1}_H)$ for every $(H,c)$ group compactification of $G$, where $\mathrm{1}_H$ is the unit element of $H$, and
			\item $\mathrm{Rot}(\Phi) \coloneqq \Phi$ for a morphism $\Phi \colon (H_1,c_1) \rightarrow (H_2,c_2)$ of group compactifications.
		\end{enumerate}
	We call this the \emph{rotation functor}.
\end{construction}

Conversely, we use the functor $F$ from above to define a functor from the category $\mathbf{MinDisc}_{\mathbf{pt}}(G)$ of pointed minimal systems with discrete spectrum to the category $\mathbf{Comp}(G)$ of compactifications. Again we use the notation of \cref{prepcompfunct}.

\begin{construction}\label{envgrpfunct}
	We define a functor $\mathrm{Env} \colon \mathbf{MinDisc}_{\mathbf{pt}}(G) \rightarrow \mathbf{Comp}(G)$ by setting
		\begin{enumerate}[(i)]
				\item $\mathrm{Env}(K,\varphi) \coloneqq (\mathrm{E}_{\mathrm{u}}(K,\varphi), i_{(K,\varphi)})$ for every pointed minimal system $(K,\varphi,y)$ with discrete spectrum, and
				\item $\mathrm{Env}(q) \coloneqq p_q$ for every morphism $q \colon (K,\varphi,y) \rightarrow (L,\psi,z)$ of such systems.
		\end{enumerate}
	We call this the \emph{enveloping group functor}.
\end{construction}

We will now show that for a commutative group $G$ the functors $\mathrm{Rot}\colon \mathbf{Comp}(G) \rightarrow \mathbf{MinDisc}_{\mathbf{pt}}(G)$ and $\mathrm{Env} \colon \mathbf{MinDisc}_{\mathbf{pt}}(G) \rightarrow \mathbf{Comp}(G)$ are essentially inverse to each other. Hence the categories of pointed minimal systems with discrete spectrum and of group compactifications are equivalent in this case.

\begin{proposition}\label{natiso1comp}
	Assume that $G$ is abelian. For every pointed minimal system $(K,\varphi,y)$ with discrete spectrum $\gamma_{(K,\varphi,y)} \colon \mathrm{Rot}(\mathrm{Env}(K,\varphi,y)) \rightarrow (K,\varphi,y), \, \vartheta \mapsto \vartheta(y)$ is an isomorphism. Moreover, these isomorphisms define a natural isomorphism 
		\begin{align*}
			\gamma \colon \mathrm{Rot} \circ \mathrm{Env} \rightarrow \mathrm{Id}_{\mathbf{MinDisc}_{\mathbf{pt}}(G)}.
		\end{align*}
\end{proposition}
\begin{proof}
	Let $(K,\varphi,z)$ be a  pointed minimal system with discrete spectrum. Then $\mathrm{E}_{\mathrm{u}}(K,\varphi)$ is a compact abelian group acting effectively and transitively on $K$. Such an action is necessarily free. Using that the action is free and transitive, we conclude that $\gamma_{(K,\varphi,y)}$ is an isomorphism. It is readily checked that $\gamma$ is actually a natural isomorphism. 
\end{proof}

As a consequence of \cref{natiso1comp} we obtain that, if $G$ is commutative, the difference between the categories $\mathbf{MinDisc}(G)$ and $\mathbf{MinDisc}_{\mathbf{pt}}(G)$ is quite inessential.

\begin{corollary}\label{pointedvsnonpointed2}
	Let $G$ be abelian and let $(K,\varphi,y)$ and $(L,\psi,z)$ be pointed minimal system with discrete spectrum. Then $(K,\varphi,y)$ and $(L,\psi,z)$ are isomorphic if and only if $(K,\varphi)$ and $(L,\psi)$ are isomorphic.
\end{corollary}
\begin{proof}
	If $(K,\varphi)$ and $(L,\psi)$ are isomorphic, then also the respective group compatifications $(\mathrm{E}_{\mathrm{u}}(K,\varphi), i_{(K,\varphi)})$ and $(\mathrm{E}_{\mathrm{u}}(L,\psi), i_{(L,\psi)})$ are isomorphic. But then $(K,\varphi,y)$ and $(L,\psi,z)$ are isomorphic by \cref{natiso1comp}.
\end{proof}

We now turn to the natural isomorphism between $\mathrm{Id}_{\mathbf{Comp}(G)}$ and the composed functor $\mathrm{Env} \circ \mathrm{Rot} \colon \mathbf{Comp}(G) \rightarrow \mathbf{Comp}(G)$.

\begin{proposition}\label{natiso2comp}
	For a group compactification $(H,c)$ we obtain an isomorphism
		\begin{align*}
			\eta_{(H,c)}\colon (H,c) \rightarrow \mathrm{Env}(\mathrm{Rot}(H,c)), \quad y \mapsto \vartheta_y
		\end{align*}
	where $\vartheta_y(x) = yx$ for $x,y \in H$. Moreover, these isomorphisms define a natural isomorphism 
		\begin{align*}
			\mathrm{Id}_{\mathbf{Comp}(G)} \rightarrow \mathrm{Env} \circ \mathrm{Rot}.
		\end{align*}
\end{proposition}
The proof is straightforward and omitted. Combining \cref{natiso1comp} and \cref{natiso2comp} we obtain the following concise version of the representation part of the topological Halmos-von Neumann theorem \cref{comhvn1} (iii).

\begin{theorem}\label{equivcomp}
	For an abelian group $G$ the functors
		\begin{align*}
			\mathrm{Env} &\colon\mathbf{MinDisc}_{\mathbf{pt}}(G)\rightarrow \mathbf{Comp}(G),\\
			\mathrm{Rot} &\colon  \mathbf{Comp}(G) \rightarrow \mathbf{MinDisc}_{\mathbf{pt}}(G)
		\end{align*}
	establish an equivalence between the category of pointed minimal topological dynamical systems with discrete spectrum and the category of group compactifications.
\end{theorem}

We now turn our attention to the uniqueness and realization aspects of the topological Halmos-von Neumann theorem for commutative groups. For this we introduce the following category. 

\begin{definition}\label{dualsbgrp}
	Assume that $G$ is commutative. The \emph{category of dual subgroups} has as objects the subgroups of $G^*$. Moreover, the set of morphisms for subgroups $\sigma,\rho \subset G^*$ is given by $\mathrm{Mor}(\sigma,\rho) \coloneqq\{i_{\sigma,\rho}\}$ for the inclusion map $i \colon \sigma \rightarrow \rho$ if $\sigma$ is a subgroup of $\rho$, and $\mathrm{Mor}(\sigma,\rho) \coloneqq \emptyset$ otherwise. We denote this category by $\mathbf{Subgrp}(G^*)$.
\end{definition}

Since the point spectrum of a minimal dynamical system is a subgroup of $G^*$ by \cref{subgroup1}, we already have a canonical way to assign to every object of $\mathbf{MinDisc}_{\mathbf{pt}}(G)$ an object of $\mathbf{Subgrp}(G^*)$. In fact, this defines a functor.

	\begin{construction}\label{functor1}
		For abelian $G$ we define a functor $\sigma_{\mathrm{p}} \colon \mathbf{MinDisc}_{\mathbf{pt}}(G) \rightarrow \mathbf{Subgrp}(G^*)^{\mathrm{op}}$ by setting
			\begin{itemize}
				\item $\sigma_{\mathrm{p}}(K,\varphi,y) \coloneqq \sigma_{\mathrm{p}}(K,\varphi)$ for every pointed minimal system $(K,\varphi,y)$ with discrete spectrum, and
				\item $\sigma_{\mathrm{p}}(q) \coloneqq i_{\sigma_{\mathrm{p}}(L,\psi),\sigma_{\mathrm{p}}(K,\varphi)}$ for every morphism $q \colon (K,\varphi,y) \rightarrow (L,\psi,z)$. 
			\end{itemize}
		We call this the \emph{point spectrum functor}.
	\end{construction}
	Observe here that if $q \colon (K,\varphi) \rightarrow (L,\psi)$ is an extension of topological dynamical systems, then then all eigenspaces of $T^\psi$ embed into eigenspaces of $T^\varphi$ via $J_q$ which implies the inclusion $\sigma_{\mathrm{p}}(L,\psi) \subset \sigma_{\mathrm{p}}(K,\varphi)$.

	To show that the functor $\sigma_{\mathrm{p}}$ establishes an equivalence between the categories $\mathbf{MinDisc}_{\mathbf{pt}}(G)$ and $\mathbf{Subgrp}(G^*)^{\mathrm{op}}$, we rewrite it as a decomposition of two functors. The first one is the enveloping group functor $\mathrm{Env} \colon\mathbf{MinDisc}_{\mathbf{pt}}(G)\rightarrow \mathbf{Comp}(G)$ from \cref{envgrpfunct}. The second functor is the following.

	\begin{construction}\label{discreteduality}
		For abelian $G$ we define a functor $\mathrm{DDual} \colon \mathbf{Comp}(G) \rightarrow \mathbf{SubGrp}(G^*)^{\mathrm{op}}$ by setting
			\begin{itemize}
				\item $\mathrm{DDual}(H,c) \coloneqq c^*H^* \coloneqq \{\chi \circ c \mid \chi \in H^*\}$ for every group compactifiction $(H,c)$ of $G$, and
				\item $\mathrm{DDual}(\Phi) \coloneqq i_{c_2^*H_2^*,c_1^*H_1^*}$ for every morphism $\Phi \colon (H_1,c_1) \rightarrow (H_2,c_2)$ of group compactifications of $G$.
			\end{itemize}
		We call this the \emph{discrete duality functor}.
	\end{construction}
	Note here that for every morphism $\Phi \colon (H_1,c_1) \rightarrow (H_2,c_2)$ of group compactifications the inclusion $c_2^*H_2^* \subset c_1^*H_1^*$ actually holds.

	\begin{proposition}\label{decompfunct}
		For commutative $G$ the diagram of functors 
			\[
			\xymatrix{
					\mathbf{MinDisc}_{\mathbf{pt}}(G) \ar[rr]^{\sigma_{\mathrm{p}}} \ar[rd]_{\mathrm{Env}} & & \mathbf{SubGrp}(G^*)^{\mathrm{op}} \\
					& \mathbf{Comp}(G) \ar[ru]_{\mathrm{DDual}} &  
				}
		\]	
		commutes.
	\end{proposition}
	\begin{proof}
		We show that the diagram commutes on the level of objects. It is then clear that it also commutes for morphisms. So let $(K,\varphi,z)$ be a pointed minimal system with discrete spectrum. We have to show that $\mathrm{DDual}(\mathrm{Env}(K,\varphi,z)) = \sigma_{\mathrm{p}}(K,\varphi,z)$. 
		
	To show prove the inclusion \enquote{$\subset$}, let $q\coloneqq \eta_{(K,\varphi,z)} \colon \mathrm{E}_{\mathrm{u}}(K,\varphi) \rightarrow K, \, \vartheta \mapsto \vartheta(z)$ be the isomorphism from \cref{natiso1comp}. If $\chi \in \mathrm{E}_{\mathrm{u}}(K,\varphi)^*$, then $f \coloneqq \overline{\chi} \circ q^{-1} \in \mathrm{C}(K)\setminus \{0\}$ satisfies 
		\begin{align*}
			T_{t}^\varphi f(x) = \overline{\chi}(q^{-1}(\varphi_{t}^{-1}(x))) = \overline{\chi}(\varphi_{t}^{-1} \circ q^{-1}(x)) = \chi(\varphi_t)\overline{\chi}(q^{-1}(x)) = \chi(i_{(K,\varphi)}(t)) f(x)
		\end{align*}
	for every $x \in K$ and $t \in G$. Thus, $\chi \circ i_{(K,\varphi)} \in \sigma_{\mathrm{p}}(K,\varphi)$.
	
	 For the converse inclusion  \enquote{$\supset$} take $\chi \in \sigma_{\mathrm{p}}(K,\varphi)$ and pick $f \in \mathrm{C}(K)\setminus\{0\}$ such that $T_t^\varphi f = \chi(t) f$ for every $t \in G$. By compactness of the unit circle $\T$, we find for every $\vartheta \in \mathrm{E}_{\mathrm{u}}(K,\varphi)$ some $\lambda_{\vartheta} \in\T$ with $f \circ \vartheta  = \lambda_{\vartheta}f$. Since $f \neq 0$, we obtain that $\lambda_{\vartheta}$ is uniquely determined by $\vartheta$. It is easy to see that the map 
	 	\begin{align*}
	 		\tilde{\chi}\colon \mathrm{E}_{\mathrm{u}}(K,\varphi) \rightarrow \T, \quad \vartheta \mapsto \lambda_{\vartheta}
	 	\end{align*}
	 is a continuous character of $\mathrm{E}_{\mathrm{u}}(K,\varphi)$ with $\tilde{\chi} \circ i_{(K,\varphi)} = \chi$.
	\end{proof}
	In view of \cref{equivcomp} we thus know that the point spectrum functor $\sigma_{\mathrm{p}} \colon \mathbf{MinDisc}(G) \rightarrow \mathbf{SubGrp}(G^*)^{\mathrm{op}}$ defines an equivalence of categories if we find an essential inverse for $\mathrm{DDual} \colon \mathbf{Comp}(G) \rightarrow \mathbf{SubGrp}(G^*)^{\mathrm{op}}$. For this we prove the following lemma which is basically a consequence of the Pontryagin duality theory.
	\begin{lemma}\label{complemma}
		For commutative $G$ and a subgroup $\sigma \subset G^*$of $G^*$ consider
			\begin{enumerate}[(i)]
				\item the dual group $\sigma^*$ with respect to the topology of pointwise convergence on $\sigma$, and
				\item the evaluation map $c_\sigma \colon G \rightarrow \sigma^*$ given by $c_\sigma(t)(\chi) \coloneqq \chi(t)$ for $t \in G$ and $\chi \in \sigma$. 
			\end{enumerate}
		Then $(\sigma^*,c_\sigma)$ is a group compactification of $G$.
	\end{lemma}
	\begin{proof}
		It is clear that $\sigma^*$ is a compact group and that $c_\sigma$ is a continuous group homomorphism. We show that $c_\sigma$ has dense range, i.e., that the closure $H \coloneqq \overline{c_\sigma(G)}$ equals $\sigma^*$. However, as $H$ and $\sigma^*$ are compact groups, this is the case if and only if the restriction map
		\begin{align*}
			r \colon (\sigma^{*})^* \rightarrow H^*, \quad \alpha \mapsto \alpha|_H
		\end{align*}
	is injective, see \cite[Theorem 4.39]{Foll2016}. By Pontryagin duality we can identify $\sigma$ with $(\sigma^*)^*$ as a discrete group via the canonical evaluation map. Therefore $r$ is injective if and only if
		\begin{align*}
			\tilde{r}\colon \sigma \rightarrow H^*, \quad \chi \mapsto j_{\chi}
		\end{align*}
		with $j_\chi(x) = x(\chi)$ for $x \in H$ and $\chi \in \sigma$ is injective. Now take $\chi \in \ker(\tilde{r})$. Then, in particular, 
		\begin{align*}
			\chi(t) = c_\sigma(t)(\chi) =1 
		\end{align*}
	for every $t \in G$, and therefore $\chi = 1$. This finishes the proof of the lemma.
	\end{proof}
	With the notation of \cref{complemma} we now construct the desired essential inverse for the discrete duality functor $\mathrm{DDual} \colon \mathbf{Comp}(G) \rightarrow \mathbf{SubGrp}(G^*)^{\mathrm{op}}$.
	\begin{construction}
		For commutative $G$ we define a functor $\mathrm{CDual}  \colon \mathbf{SubGrp}(G^*)^{\mathrm{op}} \rightarrow \mathbf{Comp}(G)$ by setting
			\begin{itemize}
				\item $\mathrm{CDual}(\sigma) \coloneqq (\sigma^*,c_\sigma)$ for every subgroup $\sigma $ of $G^*$, and
				\item $\mathrm{CDual}(i_{\sigma_1,\sigma_2}) \colon \mathrm{CDual}(\sigma_2) \rightarrow \mathrm{CDual}(\sigma_1), \, \chi \mapsto \chi|_{\sigma_1}$ for two subgroups $\sigma_1, \sigma_2\subset G^*$ with $\sigma_1 \subset \sigma_2$.
			\end{itemize}

		We call this the \emph{compact duality functor}.
	\end{construction}
	To show that the discrete and compact duality functors are essentially inverse to each other, we again have to find two natural isomorphisms. 
	\begin{proposition}\label{dualiso1}
		Assume that $G$ is commutative. For a group compactification $(H,c)$ we obtain an isomorphism
			\begin{align*}
				\gamma_{(H,c)}\colon (H,c) \rightarrow \mathrm{CDual}(\mathrm{DDual}(H,c))
			\end{align*}
		by setting $\gamma_{(H,c)}(x)(\chi \circ c) \coloneqq \chi(x)$ for every $\chi \in H^*$ and $x \in H$. Moreover, $\gamma$ defines a natural isomorphism 
			\begin{align*}
				\gamma \colon \mathrm{Id}_{\mathbf{Comp}(G)} \rightarrow \mathrm{CDual} \circ \mathrm{DDual}.
			\end{align*} 
	\end{proposition}
	\begin{proof}
		Take a group compactification $(H,c)$ and note that $\gamma_{(H,c)}(x)$ is well-defined for every $x \in H$, since $c$ has dense range. It is clear that $\gamma_{(H,c)}$ is a morphism of compactifications and hence surjective (cf. \cref{morphsurj}). On the other hand, it is injective since $H^*$ separates the points of $H$. This shows that $\gamma_{(H,c)}$ is an isomorphism of group compatifications. It is readily checked that $\gamma$ defines a natural isomorphism.
	\end{proof}
	\begin{proposition}\label{dualiso2}
		If $G$ is commutative and $\sigma \subset G^*$ is a subgroup of $G^*$, we obtain $\sigma = \mathrm{DDual}(\mathrm{CDual}(\sigma))$. In particular, the identity on every such subgroup defines a natural isomorphism 
			\begin{align*}
				\gamma \colon \mathrm{Id}_{\mathbf{SubGrp}(G^*)^{\mathrm{op}}} \rightarrow \mathrm{DDual} \circ \mathrm{CDual}.
			\end{align*}
	\end{proposition}
	\begin{proof}
		Take $\chi \in \sigma$. Then $\alpha_\chi \in (\sigma^*)^*$ with $\alpha_\chi(x) = x(\chi)$ for all $x \in \sigma^*$ satisfies
			\begin{align*}
				(\alpha_\chi \circ c_\sigma)(t) = \alpha_\chi(c_\sigma(t)) = c_\sigma(t)(\chi) = \chi(t)
			\end{align*}
		for all $t \in G$. Thus, $\chi = \alpha_\chi \circ c_\sigma \in \mathrm{DDual}(\mathrm{CDual}(\sigma))$. 
		
		Conversely, take $\alpha \in (\sigma^*)^*$. By the Pontryagin duality we find $\chi \in \sigma$ with $\alpha = \alpha_\chi$ (with $\alpha_\chi$ defined as above) and by the computation above we conclude that $\alpha \circ c_\sigma = \chi \in \sigma$. 
	\end{proof}

	We combine the natural isomorphisms of \cref{dualiso1} and \cref{dualiso2} with \cref{equivcomp} and \cref{decompfunct} to obtain the following categorical version of the topological Halmos von Neumann theorem which shows that for a commutative group $G$ the categories
		\begin{itemize}
			\item $\mathbf{MinDisc}_{\mathbf{pt}}(G)$ of pointed minimal systems with discrete spectrum,
			\item $\mathbf{Comp}(G)$ of group compactifications, and
			\item $\mathbf{SubGrp}(G^*)^{\mathrm{op}}$ of dual subgroups,
		\end{itemize}
	are pairwise equivalent. 
	\begin{theorem}\label{hvncat1}
		Let $G$ be commutative.
			\begin{enumerate}[(i)]
				\item The functors
					\begin{align*}
						\mathrm{Env} &\colon\mathbf{MinDisc}_{\mathbf{pt}}(G)\rightarrow \mathbf{Comp}(G),\\
						\mathrm{Rot} &\colon  \mathbf{Comp}(G) \rightarrow \mathbf{MinDisc}_{\mathbf{pt}}(G)
					\end{align*}
				establish an equivalence between the category of pointed minimal systems with discrete spectrum and the category of group compactifications.
				\item \label{hvncat12} The functors
					\begin{align*}
						\mathrm{DDual} &\colon \mathbf{Comp}(G) \rightarrow \mathbf{SubGrp}(G^*)^{\mathrm{op}},\\
						\mathrm{CDual} &\colon \mathbf{SubGrp}(G^*)^{\mathrm{op}} \rightarrow \mathbf{Comp}(G) 
					\end{align*}
				establish an equivalence between the category of group compactifications and the opposite category of dual subgroups.
				\item The functor 
					\begin{align*}
						\sigma_{\mathrm{p}} = \mathrm{DDual} \circ \mathrm{Env} \colon \mathbf{MinDisc}_{\mathbf{pt}}(G) \rightarrow \mathbf{SubGrp}(G^*)^{\mathrm{op}}
					\end{align*}
				defines an equivalence between the category of pointed minimal systems with discrete spectrum and the opposite category of dual subgroups.
			\end{enumerate}
	\end{theorem}
	Note that in view of \cref{pointedvsnonpointed2} the classical topological version \cref{comhvn1} (and hence also the ergodic theoretic version \cref{comhvn2}) is a direct corollary of \cref{hvncat1}. Observe also that via \cref{hvncat12} the whole dual $G^*$, seen as a subgroup of itself, gives rise to a maximal compactification of $G$, known as the \emph{Bohr compactification} (cf. \cite[Section 4.7]{Foll2016}).

\section{Halmos-von Neumann for general groups}\label{scatgen}

We now turn to minimal actions with discrete spectrum of an arbitrary group $G$. This situation is significantly more intricate for two reasons: 

	\begin{enumerate}[(i)]
		\item The representation theory for non-abelian compact groups is more complicated than in the abelian case.
		\item A transitive and effective action of a non-abelian compact group need not be free. As a consequence, not every minimal system with discrete spectrum is isomorphic to a group rotation.\footnote{However, it is still isomorphic to a quasi-rotation, see \cref{finalrem} (iii) below.}
	\end{enumerate}

We will circumvent this second difficulty by restricting to so-called \emph{normal actions} bringing us back to the case of rotation systems induced by group compactifications of $G$. To find a meaningful isomorphism invariant for such systems, we first recall some standard concepts from representation theory.

\subsection{Preliminaries on finite-dimensional representations}\label{scatgen1} The following definitions are standard (see, e.g., \cite[Part II]{Kiri1976}). 
\begin{definition}
	A bounded continuous\footnote{Boundedness and continuity refer to the unique Hausdorff topological vector space structure on $\mathscr{L}(E_\pi)$. For example, this is induced by the operator norm arising from any norm on $E_T$.} group representation $\pi\colon G\rightarrow \mathscr{L}(E_\pi), \, t \mapsto \pi_t$ of $G$ on a finite-dimensional vector space $E_\pi$ is called a \emph{finite-dimensional representation} of $G$. It is \emph{irreducible} if it is non-zero and $\{0\}$ and $E_\pi$ are the only invariant subspaces of $E_\pi$.
	
		A \emph{morphism} $U \colon \pi \rightarrow \varrho$ between such finite-dimensional representations is a linear map $U \in \mathscr{L}(E_\pi,E_\varrho)$ between the underlying vector spaces such that the diagram
		\[
			\xymatrix{
					E_\pi \ar[d]^{U} \ar[r]^{\pi_t} & E_\pi \ar[d]^{U} \\
					E_\varrho \ar[r]^{\varrho_t}& E_\varrho
			}
		\]	
		commutes for every $t \in G$. 
\end{definition}

To avoid set-theoretical issues, we restrict to a \emph{set} of representations in the following (cf. \cite[footnote 2 on page 2]{HeRo1970}). More precisely, we fix a set $M$ of finite-dimensional vector spaces with the following properties.
	\begin{enumerate}[(i)]
		\item $\C^n \in M$ for every $n \in \N_0$.
		\item For every two spaces $E,F \in M$ there is a direct sum $E \oplus F$ in $M$,
		\item For every two spaces $E,F \in M$ there is a tensor product $E \otimes F$ in $M$.
		\item For every $E \in M$ the complex conjugate space $\overline{E}$ is in $M$ where $\overline{E}$ is defined as the set $E$ with the unique vector space structure turning the identity $\mathrm{id}_{E}\colon  E \rightarrow \overline{E}$ into an antilinear isomorphism.
		\item For every $E \in M$ and every linear subspace $F \subset E$, we also have $F \in M$.
	\end{enumerate}
The existence of such a set $M$ is readily verified. Having fixed such an $M$ once and for all, we write $\mathbf{FinVect}$ for the category of finite-dimensional vector spaces in $M$, $\mathbf{FinRep}(G)$ for the category of finite-dimensional representations of $G$ on vector spaces in $M$.

The following key result about morphisms between irreducible representations is known as Schur's lemma (see, e.g. \cite[Theorem 27.9 and Corollary 27.10]{HeRo1970}).
\begin{proposition}\label{schur}
		For irreducible representations $\pi$ and $\varrho$ of $G$ the following assertions hold.
			\begin{enumerate}[(i)]
				\item Every morphism $U \colon \pi \rightarrow \varrho$ is zero or an isomorphism.
				\item The vector space of morphisms $\mathrm{Mor}(\pi,\varrho)$ is at most one-dimensional.
			\end{enumerate}			 
\end{proposition}

We need some standard constructions for finite-dimensional representations. Subrepresentations of a given finite-dimensional representation as well as the direct sum $\pi \oplus \varrho$ and tensor product $\pi \otimes \varrho$ of finite-dimensional representations $\pi$ and $\varrho$ are defined in the obvious way. In addition we recall the following. 
\begin{definition}\label{complconj}
	Let $\pi$ be a finite-dimensional representation of $G$. The \emph{complex conjugate representation} of $\pi$ is given by 
			\begin{align*}
				\overline{\pi} \colon G \rightarrow \mathscr{L}(\overline{E_\pi}), \, \quad t \mapsto \pi_t.
			\end{align*}
\end{definition}

We now introduce the dual of $G$ generalizing the Pontryagin dual $G^*$ from the abelian case.

\begin{definition}\label{dual}
	Let $\mathrm{Irred}(G)$ be the set of all irreducible finite-dimensional representations within $\mathbf{FinRep}(G)$
%		\begin{align*}
%			 \coloneqq \{\pi \colon G \rightarrow \mathscr{L}(E_{\pi})\mid \pi \textrm{ irreducible finite-dimensional representation}\} \subset \mathrm{FinRep}(G)
%		\end{align*}
	and define an equivalence relation $\sim$ on $\mathrm{Irred}(G)$ via
		\begin{align*}
			\pi \sim \varrho \Leftrightarrow \textrm{ there is an isomorphism } U \colon \pi \rightarrow \varrho.
		\end{align*}
	The set $\hat{G}\coloneqq \mathrm{Irred}(G)/\sim$ is the \emph{dual of $G$}.
\end{definition}

If $\pi \colon G \rightarrow \mathscr{L}(E_\pi)$ is any irreducible finite-dimensional representation (with $E_\pi$ not necessarily contained in our fixed set $M$ of finite-dimensional vector spaces from above), then $E_\pi$ is isomorphic to a vector space $\C^n$ for some $n \in \N$. Therefore, $\pi$ is isomorphic to an irreducible representation $\tilde{\pi}$ on $\C^n \in M$ and we write $[\pi] \coloneqq [\tilde{\pi}]$.
%\footnote{One may think of $[\pi]$ as the \enquote{equivalence class} of $\pi$ within \emph{all} irreducible representations respect to equivalence of representations. However, since the class of all finite-dimensional representations of $G$ is not a set, we formally define $[\pi]$ differently to avoid set-theoretic issues (cf. \cite[footnote 2 on page 2]{HeRo1970}).}

\begin{remark}\label{abeliandual}
	For abelian $G$, the character space $G^*$ of $G$ maps bijectively to the dual $\hat{G}$ of $G$ by mapping every character to the one-dimensional representation $G\rightarrow \mathscr{L}(\C),\, t \mapsto \chi(t)$ as a consequence of Schur's lemma (cf. \cite[Corollary 3.6]{Foll2016}\footnote{The cited result assumes $G$ to be locally compact, but the proof still works in our situation.}).
\end{remark}

For locally compact groups it is common to consider unitary representations on Hilbert spaces. The following important result allows to allows to turn any finite-dimensional representation into a unitary representation (see \cite[Lemma 7.1.1]{DeEc2009}\footnote{\label{note1}The quoted result treats the case of compact groups. However, since the closure of the image of a given finite-dimensional representation (which is bounded by definition) of $G$ is compact, the lemma also holds in general.}).

\begin{lemma}\label{prepgroup}
	Let $\pi$ be a finite-dimensional representation. 
		\begin{enumerate}[(i)]
			\item There is an inner product $(\,\cdot \mid \cdot\, )$ on $E_\pi$ with respect to which $\pi$ becomes a unitary group representation.
			\item If $\pi$ is irreducible and unitary with respect to two inner products $(\,\cdot \mid \cdot\, )_1$ and $(\,\cdot \mid \cdot\, )_2$ on $E_\pi$, then there is $c > 0$ with $(\,\cdot \mid \cdot\, )_2 = c \cdot (\,\cdot \mid \cdot\, )_1$.
		\end{enumerate}
\end{lemma}

We from now on choose an invariant inner product whenever necessary. Note that in case of irreducible representations there is no ambiguity when refering to orthogonality or the group of unitary operators as these notions do not depend on the particular invariant inner product. We also note the following. 
\begin{remark}
	In case of a compact group $G$ our definition of the dual is equivalent to the usual definition of the unitary dual as in \cite[Definition 27.3]{HeRo1970} (use \cref{prepgroup}). Note however, that for a locally compact group $G$ there is a significant difference since we only consider irreducible \emph{finite-dimensional} representations here.
\end{remark}

Even for finite groups $G$, one cannot reconstruct $G$ from its dual $\hat{G}$ and instead has to work with a larger category (see the remarks in the prelude of \cite[Paragraph 30]{HeRo1970}). However, for the purpose of classifying \emph{group compactifications} of $G$, considering the dual is sufficient. 

\subsection{Grouplike subsets}\label{scatgen2} In view of \cref{hvncat1} (ii), compactifications of an abelian group $G$ are in correspondence with subgroups of $G^*$. In the general situation, the dual $\hat{G}$ is no longer a group, but we still have a (well-defined!) \enquote{multivalued multiplication}
	\begin{align*}
		\otimes \colon \hat{G} \times \hat{G} \rightarrow \mathcal{P}(\hat{G}),\quad ([\pi],[\varrho]) \mapsto \{[\tau] \in \hat{G}\mid \tau \textrm{ is a subrepresentation of } \pi \otimes \varrho\},
	\end{align*}
cf. \cite[Definition 27.35]{HeRo1970}. For $[\pi]$, $[\tau] \in \hat{G}$ the set $[\pi] \otimes [\tau]$ actually only has finitely many elements as shown by the following consequence of \cite[Theorem 27.30]{HeRo1970}\footnote{See footnote \ref{note1} above.}.

\begin{proposition}\label{irreddecomp}
 	Let $\pi$ be a finite-dimensional unitary representation. Then there are pairwise orthogonal invariant subspaces $E_1, \dots, E_n \subset E_\pi$ such that
			\begin{enumerate}[(i)]
				\item $E_\pi = E_1 \oplus \dots \oplus E_n$, and
				\item $\pi^j\colon G \rightarrow \mathscr{L}(E_j),\, t \mapsto \pi_t|_{E_j}$ is irreducible for every $j \in \{1, \dots, n\}$.
			\end{enumerate}
	Moreover, the set $\{[\pi^j]\mid j \in \{1, \dots n\}\}$ is the same for any such decomposition.
\end{proposition}
\begin{corollary}
	For $[\pi]$, $[\tau] \in \hat{G}$ the set $[\pi] \otimes [\tau] \subset \hat{G}$ is finite.
\end{corollary}
\begin{proof}
	Choose an invariant inner product on $E_{\pi \otimes \varrho}$. Since any irreducible invariant subspace of $E_{\pi \otimes \varrho}$ can be orthogonally complemented (and hence is part of an orthogonal decomposition of $E_{\pi \otimes \varrho}$ into irreducible subspaces), the claim follows directly from the uniqueness assertion of \cref{irreddecomp}.
\end{proof}
Setting $\overline{[\pi]} \coloneqq [\overline{\pi}]$ for $[\pi] \in \hat{G}$, the complex conjugate representation from \cref{complconj} also gives rise to a replacement of the inverse in a group. Finally, the element $[\mathbbm{1}] \in \hat{G}$ where $\mathbbm{1} \colon G \rightarrow \mathscr{L}(\C), \, t \mapsto \mathrm{id}_{\C}$ is the constant representation can be seen as a \enquote{neutral element} for $\hat{G}$.\footnote{\cite[Theorem 27.38]{HeRo1970} provides some more justification for these replacements of multiplication, inverses and the neutral element.}

With these concepts we arrive of the following generalized version of dual subgroups (cf. \cite[Definition 27.35]{HeRo1970}).

\begin{definition}
	A subset $\sigma \subset \hat{G}$ is \emph{grouplike} if
		\begin{enumerate}[(i)]
			\item $[\pi] \otimes [\varrho] \subset \sigma$ for all $[\pi]$, $[\varrho] \in \sigma$,
			\item $\overline{[\pi]} \in \sigma$ for every $[\pi] \in \sigma$, and
			\item $[\mathbbm{1}] \in \sigma$.
		\end{enumerate}
				
	We write $\mathbf{GrpLike}(\hat{G})$ for the \emph{category of grouplike subsets} of $\hat{G}$ where the morphisms are given by inclusion maps $i_{\sigma_1,\sigma_2}$ for $\sigma_1 \subset \sigma_2$ as in \cref{dualsbgrp}.
\end{definition}

For a compact group the following result, which is part of \cite[Theorem 27.39]{HeRo1970}, shows that there is only one grouplike subset \enquote{separating the points}.

\begin{proposition}\label{pointseparating}
	Let $G$ be compact and $\sigma \subset \hat{G}$ a grouplike subset. Assume that for every $1 \neq t \in G$ there is an irreducible representation $\pi$ of $G$ with $\pi_t \neq \mathrm{id}_{E_\pi}$ and $[\pi] \in \sigma$. Then $\sigma = \hat{G}$.
\end{proposition}

In complete analogy to the abelian case, every group compactification of $G$ gives rise to a grouplike subset of $\hat{G}$ in a functorial way (cf. \cref{discreteduality} above).

\begin{construction}\label{represfunct}	
	Define a functor $\mathrm{Rep}\colon \mathbf{Comp}(G) \rightarrow \mathbf{GrpLike}(\hat{G})^{\mathrm{op}}$ by setting	
		\begin{enumerate}[(i)]
			\item $\mathrm{Rep}(H,c) \coloneqq c^*\hat{H} \coloneqq \{[\pi \circ c]\mid [\pi] \in \hat{H}\}$ for every group compactification $(H,c)$, and 
			\item $\mathrm{Rep}(\Phi) \coloneqq i_{c_2^*\hat{H_2}, c_1^*\hat{H_1}}$ for every morphism $\Phi \colon (H_1, c_1) \rightarrow (H_2,c_2)$ of group compactifications of $G$.
		\end{enumerate}
	We call this the \emph{representations functor}.
\end{construction}

\subsection{Tannaka categories and Tannaka groups}\label{scatgen3}

Our task is now the construction of an essential inverse for the functor $\mathrm{Rep}$. This is done via the Tannaka-Krein duality (see, e.g., \cite[Paragraph 30]{HeRo1970}, \cite[Section 12.2]{Kiri1976}, \cite[Chapter 9]{Robe1983} and \cite{JoSt1991}). Our approach is mostly based on the elegant category theoretic exposition of the duality given in \cite{JoSt1991}. Before we assign a group compactification to a given grouplike subset $\sigma \subset \hat{G}$, we first build a category of representations out of $\sigma$.

\begin{construction}\label{tannakacat}
	Let $\sigma \subset \hat{G}$ be a grouplike subset. We consider the preimage $p^{-1}(\sigma)$ under the canonical map
		\begin{align*}
			p\colon \mathrm{Irred}(G) \rightarrow \hat{G}, \quad \pi \mapsto [\pi]
		\end{align*}
	and the smallest set $\mathscr{T}(\sigma)$ of representations in $\mathbf{FinRep}(G)$ containing $p^{-1}(\sigma)$ with the following properties.
	\begin{enumerate}[(i)]
		\item If $\pi, \varrho \in \mathscr{T}(\sigma)$, then $\pi \oplus \varrho\in \mathscr{T}(\sigma)$.
		\item If $\pi, \varrho \in \mathscr{T}(\sigma)$, then $\pi \otimes \varrho \in \mathscr{T}(\sigma)$.
		\item If $\pi \in \mathscr{T}(\sigma)$, then $\overline{\pi} \in \mathscr{T}(\sigma)$.
		\item If $\pi \in \mathscr{T}(\sigma)$ and $\varrho$ is a subrepresentation of $\pi$, then also $\varrho \in \mathscr{T}(\sigma)$.
	\end{enumerate}
	We consider $\mathscr{T}(\sigma)$ as a full subcategory of the category $\mathbf{FinRep}(G)$ of finite-dimensional representations of $G$ and call this the \emph{Tannaka category} associated to $\sigma$. 
\end{construction}

The following simple observation is important.
\begin{lemma}\label{irreduciblescontained}
	Let $\sigma \subset \hat{G}$ be a grouplike subset. If $\pi \in \mathscr{T}(\sigma)$ is irreducible, then $[\pi] \in \sigma$.
\end{lemma}
\begin{proof}
	We consider the set $\mathscr{T}$ of all representations $\tau \in \mathscr{T}(\sigma)$ having the property, that every irreducible subrepresentation $\pi$ of $\tau$ satisfies $[\pi] \in \sigma$. We check that $\mathscr{T}$ has properties (i) -- (iv) of \cref{tannakacat} and hence $\mathscr{T} = \mathscr{T}(\sigma)$. Note first that property (iv) holds trivially. Moreover, (i) is a consequence of  \cref{irreddecomp} showing that every irreducible subrepresentation of $\pi \oplus \varrho$ is isomorphic to an irreducible subrepresentation of $\pi$ or $\varrho$.
	Using \cref{irreddecomp} once again to decompose representations into irreducible ones, it is enough to show (ii) and (iii) for irreducible representations. But for these it is evident since $\sigma$ is grouplike.
\end{proof}

From the Tannaka category $\mathscr{T}(\sigma)$ of a grouplike subset $\sigma$ we now contruct the desired group compactification. To do so, we consider certain natural transformations of the forgetful functor 
	\begin{align*}
		\mathrm{For}\colon \mathscr{T}(\sigma) \rightarrow \mathbf{FinVect}
	\end{align*}
which sends a representation to the underlying vector space and a intertwining linear map to itself (cf. \cite[Paragraph 1]{JoSt1991}).

\begin{definition}
	Let $\sigma \subset \hat{G}$ be a grouplike subset, $\mathscr{T}(\sigma)$ the corresponding Tannaka category and $\mathrm{For} \colon \mathscr{T}(\sigma) \rightarrow \mathbf{FinVect}$ the forgetful functor. A natural transformation $\eta \colon \mathrm{For} \rightarrow \mathrm{For}$ is a \emph{Tannaka transformation} if it is
		\begin{enumerate}[(i)]
			\item \emph{tensor-preserving}, i.e., $\eta_{\pi \otimes \varrho} = \eta_{\pi} \otimes \eta_{\varrho}$ for all $\pi,\varrho\in \mathscr{T}(\sigma)$,
			\item \emph{self-adjoint}, i.e., $\eta_{\overline{\pi}} = \eta_\pi$ for every every $\pi \in \mathscr{T}(\sigma)$, and
			\item \emph{unital}, i.e., $\eta_{\mathbbm{1}} = \mathrm{id}_{\C}$.
		\end{enumerate}
	The set $\check{\sigma}$ of all Tannaka transformations $\eta \colon \mathrm{For} \rightarrow \mathrm{For}$ equipped with the composition of natural transformations and the topology induced by the product topology on $\prod_{\pi \in \mathscr{T}(\sigma)} \mathscr{L}(E_\pi)$ is the \emph{Tannaka group} of $\sigma$. 
\end{definition}

\begin{remark}\label{sums}
	It is easy to see that any Tannaka transformation also respects direct sums, i.e., $\sigma \subset \hat{G}$ is a grouplike subset, $\eta \in \check{\sigma}$ and $\pi, \varrho \in \mathscr{T}(\sigma)$, then $\eta_{\pi \oplus \varrho} = \eta_{\pi} \oplus \eta_{\varrho}$ (cf. \cite[Problem 12.4]{Kiri1976}).
\end{remark}

The following result shows that the Tannaka group of a grouplike subset is in fact a compact group. The proof is basically the one of \cite[Proposition 9.5]{JoSt1991}.
\begin{proposition}\label{tannakagroup}
	For every grouplike subset $\sigma \subset \hat{G}$ the Tannaka group $\check{\sigma}$ is a compact group.
\end{proposition}

\begin{proof}
	For every representation $\pi \in \mathscr{T}(\sigma)$ let $(\, \cdot\mid  \cdot \,)$ be an invariant inner product as in \cref{prepgroup}. We show that $\check{\sigma}$ is contained in the product
	\begin{align*}
		H \coloneqq \prod_{\pi \in \mathscr{T}(\sigma)} \mathrm{U}(E_\pi)
	\end{align*}
	of the corresponding unitary groups $\mathrm{U}(E_\pi)$ for $\pi \in \mathscr{T}(\sigma)$.
	
	Pick $\eta \in \check{\sigma}$. For a fixed representation $\pi \in \mathscr{T}(\sigma)$ consider the morphism of representations $V \colon \pi \otimes \overline{\pi} \rightarrow \mathbbm{1}$ given by $V(x \otimes y) \coloneqq (x|y)$ for $x, y \in E_\pi$.
			 
	Since $\eta$ is a natural transformation, we obtain $V \circ \eta_{\pi \otimes \overline{\pi}} = \eta_{\mathbbm{1}} \circ V$. Using that $\eta$ is unital and tensor-preserving, this implies
		\begin{align*}
			V \circ \eta_{\pi} \otimes \eta_{\overline{\pi}} =  \mathrm{id}_{\C} \circ V = V.
		\end{align*}
		But then
			\begin{align*}
				\left(\eta_{\pi}x\mmid \eta_{\pi}y\right) = V(\eta_{\pi}x \otimes \eta_{\overline{\pi}}y) = V(x \otimes y) = (x|y)
			\end{align*}
		for all $x,y \in E_\pi$ since $\eta$ is self-adjoint. Thus, $\eta_\pi \in \mathscr{L}(E_\pi)$ is an isometry and, since $E_\pi$ is finite-dimensional, even unitary. Thus, $\check{\sigma}$ is contained in $H$.
	
	To see that $\check{\sigma}$ is even a subgroup of $H$, obeserve first that it contains the neutral element of $H$ and is closed with respect to composition. To obtain that it is also closed with respect to inverses, let again $\eta \in \check{\sigma}$ and consider $\eta^* \in H$ given by $(\eta^*)_{\pi}
\coloneqq (\eta_\pi)^*$ for every $\pi \in \mathscr{T}(\sigma)$. If $U \colon \pi \rightarrow \varrho$ is a morphism of finite-dimensional representations, then the adjoint $U^*$ defines a morphism $\varrho \rightarrow \pi$. Thus,
	\begin{align*}
		(\eta_\varrho)^*U = (U^*\eta_\varrho)^* = (\eta_\pi U^*)^* = U (\eta_\pi)^*
	\end{align*}
	showing that $\eta^*\colon \mathrm{For} \rightarrow \mathrm{For}$ is a natural transformation. It is routine to check that $\eta^*$ satisfies properties (i) -- (iii). We therefore obtain that $\eta^* \in \check{\sigma}$ and hence $\check{\sigma}$ is a subgroup of the compact group $H$. Finally, it is easy to see that $\check{\sigma}$ is closed and therefore itself a compact group. 
\end{proof}

To obtain a group compatification of $G$, we consider the following canonical map (which substitutes the evaluation mapping from the abelian case, cf. \cref{complemma}).
\begin{definition}\label{groupcompmap}
	Let $\sigma \subset \hat{G}$ a grouplike subset and $\check{\sigma}$ its Tannaka group. We define
		\begin{align*}
			c_{\sigma} \colon G \rightarrow \check{\sigma},\quad t \mapsto \eta(t)
		\end{align*}
	by setting $\eta(t)_{\pi} \coloneqq \pi_t$ for every $\pi \in \mathscr{T}(\sigma)$ and every $t \in G$.
\end{definition}

We now want to show the following.

\begin{proposition}\label{isgroupcomp}
	For every grouplike subset $\sigma$ the pair $(\check{\sigma},c_\sigma)$ is a group compactification of $G$.
\end{proposition}

While it is clear that $c_{\sigma}$ is a continuous group homomorphism, showing that its range is dense in $\check{\sigma}$ requires some work. In the following two lemmas we first elaborate some properties of closed subgroups $H$ of $\check{\sigma}$ containing the image $c_\sigma(G)$. We will then see that $H = \check{\sigma}$ is the only such subgroup thereby finishing the proof of \cref{isgroupcomp}.

\begin{lemma}\label{equivirred}
	Let $\sigma \subset \hat{G}$ be a grouplike subset and $H \subset \check{\sigma}$ a closed subgroup of $\check{\sigma}$ containing $c_\sigma(G)$. For a representation $\pi \in \mathscr{T}(\sigma)$ of $G$ the induced representation
	\begin{align*}
			\pi_H \colon H \rightarrow \mathscr{L}(E_\pi), \quad \eta \mapsto \eta_{\pi}
	\end{align*}
	has the following properties.
		\begin{enumerate}[(i)]
			\item A subspace $E \subset E_\pi$ is $\pi$-invariant if and only if it is $\pi_H$-invariant.
			\item The representation $\pi$ is irreducible if and only if $\pi_H$ is irreducible.
		\end{enumerate}		 
\end{lemma}
\begin{proof}
	Part (ii) follows directly from part (i). To prove (i), take a subspace $E \subset E_\pi$. If $E$ is $\pi$-invariant, then the inclusion map $i \colon E \rightarrow E_\pi$ defines a morphism from the restricted representation $\pi|_E$ to the representation $\pi$. Thus, if $\eta\in H \subset \check{\sigma}$, the diagram
		\[
			\xymatrix{
					E \ar[d]_{\eta_{\pi|_E}} \ar[r]^{i} & E_\pi \ar[d]^{\eta_{\pi}} \\
					E \ar[r]^{i}& E_\pi
			}
		\]	
	commutes by the definition of natural transformations. This shows $\eta_{\pi}(E) \subset E$. We conclude that $E$ is $\pi_H$-invariant.
	
	Assume conversely that $E$ is $\pi_H$-invariant. Then, in particular, $\pi_tE = c_\sigma(t)E\subset E$ for every $t \in G$. 
\end{proof}

With the notation of \cref{groupcompmap} and \cref{equivirred} we also obtain the following.
\begin{lemma}\label{inversebij}
	Let $\sigma \subset \hat{G}$ be a grouplike subset and $H \subset \check{\sigma}$ a closed subgroup containing $c_\sigma(G)$. Then the maps
		\begin{align*}
			\delta & \colon \sigma \rightarrow \hat{H}, \quad [\pi] \mapsto [\pi_H], \\
			c_\sigma^*&\colon \hat{H} \rightarrow \sigma, \quad [\pi] \mapsto [\pi \circ c_\sigma]
		\end{align*}
	are mutually inverse bijections.
\end{lemma}

\begin{proof}
	In view of \cref{equivirred} we obtain an injective mapping	
		\begin{align*}
			\tilde{\delta}\colon \{\pi \in \mathrm{Irred}(G)\mid [\pi] \in \sigma\} \rightarrow \mathrm{Irred}(H), \quad \pi \mapsto \pi_H.
		\end{align*}
	By the properties of natural transformations, two irreducible representations $\pi$ and $\varrho$ of $G$ with $[\pi], [\varrho] \in \sigma$ are isomorphic if and only if the induced representations $\pi_H$ and $\varrho_H$ of $H$ are isomorphic. Therefore, $\tilde{\delta}$ induces an injective map
		\begin{align*}
			\delta &\colon \sigma \rightarrow \hat{H}, \quad [\pi] \mapsto [\pi_H].
		\end{align*}
	
	To show that $\delta$ is surjective, we apply \cref{pointseparating}. Take $\eta \in H$ with $\eta \neq \mathrm{id}$. Thus, there is $\pi \in \mathscr{T}(\sigma)$ with $\eta_\pi \neq \mathrm{id}_{E_\pi}$, i.e., $(\pi_H)_\eta \neq \mathrm{id}_{E_\pi}$. Using \cref{irreddecomp} and \cref{sums} we may assume that $\pi$ is irreducible,  hence $[\pi] \in \sigma$ by \cref{irreduciblescontained}. A moment's thought reveals that the image $\delta(\sigma) \subset \hat{H}$ is grouplike, hence \cref{pointseparating} shows $\delta(\sigma) = \hat{H}$.
	
	Now observe that the map
		\begin{align*}
			c_\sigma^*\colon \hat{H} \rightarrow \hat{G}, \quad [\pi] \mapsto [\pi \circ c_\sigma]
		\end{align*}
	is well-defined. Moreover, if $[\pi] \in \hat{H}$ we find $[\tau] \in \sigma$ with $[\pi] = \delta([\tau]) = [\tau_H]$. But since $\tau_H \circ c_\sigma = \tau$, we conclude that $c_\sigma^*([\pi]) \in \sigma$.  This argument also shows that $c_\sigma^*$ and $\delta$ are inverse to each other.
\end{proof}

We now prove \cref{isgroupcomp} similarly to \cite[Theorem 20]{JoSt1991}.

\begin{proof}[Proof of \cref{isgroupcomp}]
	We choose a set $R$ of representatives for the elements of $\sigma$ and for every $\pi \in R$ an invariant inner product $(\,\cdot\mid\cdot\,)_\pi$ as well as an orthonormal basis $e_1, \dots, e_{n_{\pi}}$ wih respect to $(\,\cdot\mid\cdot\,)_\pi$. By \cref{inversebij} applied to $H = \check{\sigma}$ and the Peter-Weyl theorem (see, e.g., \cite[Theorem 5.12]{Foll2016}), the weighted matrix coefficients
		\begin{align*}
			\alpha_{ij}^\pi\colon \check{\sigma} \rightarrow \C, \quad \eta \mapsto \sqrt{n_{\pi}} \cdot (\eta_{\pi} e_i|e_j)_\pi
		\end{align*}
	for $i,j \in \{1, \dots, n_{\pi}\}$ and $\pi \in R$ define an orthonormal basis of the Hilbert space $\mathrm{L}^2(\check{\sigma})$ (where $\check{\sigma}$ is equipped with the Haar measure) and span a dense subset of $\mathrm{C}(\check{\sigma})$. Similarly, applying \cref{inversebij} to $H= \overline{c_\sigma(G)}$ and using the Peter-Weyl theorem shows that the matrix coefficients
		\begin{align*}
			\beta_{ij}^\pi \colon \overline{c_\sigma(G)} \rightarrow \C, \quad \eta \mapsto \sqrt{n_{\pi}} \cdot (\eta_{\pi} e_i|e_j)_\pi
		\end{align*}
	for $i,j \in \{1, \dots, n_{\pi}\}$ and $\pi \in R$ define an orthonormal basis of the Hilbert space $\mathrm{L}^2(\overline{c_\sigma(G)})$ and span a dense subset of $\mathrm{C}(\overline{c_\sigma(G)})$. Since the restriction mapping
		\begin{align*}
			\mathrm{C}(\check{\sigma}) \rightarrow \mathrm{C}(\overline{c_\sigma(G)}), \quad  f\mapsto f|_{\overline{c_\sigma(G)}}
		\end{align*}
	sends $\alpha_{ij}^\pi$ to $\beta_{ij}^\pi$ for $i,j \in \{1, \dots, n_\pi\}$ and $\pi \in R$, it defines an isometry with respect to the $\mathrm{L}^2$-norms. In particular, it is injective. Thus, $\overline{c_\sigma(G)} = \check{\sigma}$.
\end{proof}

The construction of group compactifications from grouplike subsets of $\hat{G}$ is compatible with morphisms: If $\sigma_1, \sigma_2 \subset \hat{G}$ are grouplike subsets with $\sigma_1 \subset \sigma_2$, then also $\mathscr{T}(\sigma_1) \subset \mathscr{T}(\sigma_2)$ and the restriction map
	\begin{align*}
		\check{\sigma_2} \rightarrow \check{\sigma_1}, \quad \eta \mapsto \eta|_{\mathscr{T}(\sigma_1)}
	\end{align*}
defines a morphism of group compactifications. We therefore obtain a functor.

\begin{construction}\label{tannakafunct}
	Define a functor $\mathrm{Tan}\colon \mathbf{GrpLike}(\hat{G})^{\mathrm{op}} \rightarrow \mathbf{Comp}(G)$ by setting
		\begin{itemize}
			\item $\mathrm{Tan}(\sigma) \coloneqq \check{\sigma}$ for every grouplike subset $\sigma \subset \hat{G}$, and
			\item $\mathrm{Tan}(i_{\sigma_1,\sigma_2}) \colon \check{\sigma_2} \rightarrow \check{\sigma_1}, \, \eta \mapsto \eta|_{\mathscr{T}(\sigma_1)}$ for two grouplike subsets $\sigma_1, \sigma_2 \subset \hat{G}$ with $\sigma_1 \subset \sigma_2$.
		\end{itemize}
	We call this the \emph{Tannaka functor}.
\end{construction}

\subsection{Tannaka-Krein duality for group compactifications}\label{scatgen4}
We now show that   
	\begin{itemize}
		\item the representations functor $\mathrm{Rep}\colon \mathbf{Comp}(G) \rightarrow \mathbf{GrpLike}(\hat{G})^{\mathrm{op}}$ from \cref{represfunct}, and 
		\item the Tannaka functor $\mathrm{Tan}\colon \mathbf{GrpLike}(\hat{G})^{\mathrm{op}} \rightarrow\mathbf{Comp}(G)$ from \cref{tannakafunct}
	\end{itemize}
are essentially inverse to each other and thereby extend the categorical equivalence of \cref{hvncat1} (ii) to the case of a general (not necessarily abelian) group $G$.

To contruct the corresponding natural isomorphisms we need the following lemma. Here, $\mathscr{T}(\hat{H})$ denotes the Tannaka category of the whole dual $\hat{H}$ with respect to the compact group $H$, i.e., the smallest subcategory of $\mathbf{FinRep}(H)$ containing $\mathrm{Irred}(H)$ which is closed with respect to taking direct sums, tensor products, complex conjugate representations and subrepresentations.

\begin{lemma}\label{pullbacktannaka}
	Let $(H,c)$ be a group compactification. Then 
		\begin{align*}
			c^*\colon \mathscr{T}(\hat{H}) \rightarrow \mathscr{T}(c^*\hat{H}), \quad \pi \mapsto \pi \circ c
		\end{align*}
	is a bijection such that 
		\begin{enumerate}[(i)]
			\item $c^*(\pi \oplus \varrho) = c^*(\pi) \oplus c^*(\varrho)$ for all $\pi,\varrho \in \mathscr{T}(\hat{H})$, 
			\item $c^*(\pi \otimes \varrho) = c^*(\pi) \otimes c^*(\varrho)$ for all $\pi,\varrho \in \mathscr{T}(\hat{H})$, 
			\item $c^*(\overline{\pi}) = \overline{c^*(\pi)}$ for all $\pi \in \mathscr{T}(\hat{H})$, 
			\item given $\pi \in \mathscr{T}(\hat{H})$, a subspace $E \subset E_\pi$ is $\pi$-invariant if and only it is $c^*(\pi)$-invariant, and in that case $c^*(\pi|_E) = c^*(\pi)|_E$, and
			\item $c^*(\mathbbm{1}) = \mathbbm{1}$.
		\end{enumerate}
\end{lemma}
\begin{proof}
	For every $\pi \in \mathscr{T}(\hat{H})$ we obtain that $c^*(\pi) \coloneqq \pi \circ c\colon G \rightarrow E_\pi$ is a finite-dimensional representation of $G$. Moreover, assertions (i) -- (v) are obvious.
	In particular, the set
		\begin{align*}
			\mathscr{T} \coloneqq \{c^*(\pi)\mid \pi \in \mathscr{T}(\hat{H})\}
		\end{align*}
	satisfies conditions (i) -- (iv) of \cref{tannakacat}. Now take an irreducible representation $\pi \colon G \rightarrow \mathscr{L}(E_{\pi})$ in $\mathrm{Irred}(G)$ with $[\pi] \in c^*(\hat{H})$ and show that $\pi \in \mathscr{T}$. By assumption, we find an irreducible representation $\varrho \colon H \rightarrow \mathscr{L}(E_{\varrho})$ in $\mathrm{Irred}(H)$ and an isomorphism $U \colon \pi \rightarrow \varrho \circ c$ of finite-dimensional representations of $G$. We then obtain a new irreducible representation
		\begin{align*}
			\tau \colon H \rightarrow \mathscr{L}(E_{\pi}), \quad x \mapsto U^{-1}\varrho(x)U
		\end{align*}
	of $H$ and this satisfies $\pi = \tau \circ c = c^*(\tau) \in  \mathscr{T}$ as desired. By definition of $\mathscr{T}(c^*\hat{H})$, we conclude that $\mathscr{T}(c^*\hat{H})\subset \mathscr{T}$. 
	
	On the other hand, the set
		\begin{align*}
			\mathscr{S} \coloneqq \{\pi \in \mathscr{T}(\hat{H})\mid c^*(\pi)\in \mathscr{T}(c^*\hat{H})\}
		\end{align*}
	contains every irreducible representation $\pi \colon H \rightarrow \mathscr{L}(E_{\pi})$ in $\mathrm{Irred}(H)$ and satisfies conditions (i) -- (iv) of \cref{tannakacat}. By definition of $\mathscr{T}(\hat{H})$, we obtain $\mathscr{S} = \mathscr{T}(\hat{H})$. As a consequence, 
		\begin{align*}
			c^*\colon \mathscr{T}(\hat{H}) \rightarrow \mathscr{T}(c^*\hat{H}), \quad \pi \mapsto c^*(\pi) = \pi \circ c
		\end{align*}
	is a surjective map satisfying (i) -- (iv). Since $H$ is dense in $G$, the mapping $c^*$ is also injective and therefore a bijection. 
\end{proof}

We now prove the analogues of \cref{dualiso1} and \cref{dualiso2} establishing the desired categorical equivalence in our more general framework. With the notation of \cref{pullbacktannaka} we obtain the following.

\begin{proposition}\label{natisotannaka1}
	For a group compactification $(H,c)$ we obtain an isomorphism of group compactifications
		\begin{align*}
			\gamma_{(H,c)} \colon (H,c) \rightarrow \mathrm{Tan}(\mathrm{Rep}(H,c))
		\end{align*}
	by setting $(\gamma_{(H,c)}(x))_\pi \coloneqq ((c^*)^{-1}(\pi))_x$ for $\pi \in \mathscr{T}(c^*\hat{H})$ and $x \in H$. Moreover, $\gamma$ defines a natural isomorphism 
		\begin{align*}
			\gamma \colon \mathrm{Id}_{\mathrm{Comp}(G)} \rightarrow \mathrm{Tan} \circ \mathrm{Rep}.
		\end{align*}
\end{proposition}
\begin{proof}
	Take a group compactification $(H,c)$. Using \cref{pullbacktannaka} we obtain that for $x \in H$ the map $\gamma_{(H,c)}(x)$ is well-defined and a Tannaka transformation. It is clear, that $\gamma_{(H,c)}$ is a morphism of group compactifications and in particular a surjective map (see \cref{morphsurj}). On the other hand, since the finite-dimensional representations of a compact group separate its points (see, e.g., \cite[Theorem 3.34]{Foll2016}), we obtain that $\gamma_{(H,c)}$ is injective and hence an isomorphism of group compactifications. 
	
	To see that $\gamma$ defines a natural isomorphism, take a morphism $\Phi \colon (H_1, c_1) \rightarrow (H_2,c_2)$ of group compactifications of $G$. We abbreviate the induced restriction map
	\begin{align*}
		\mathrm{Tan}(\mathrm{Rep}(\Phi)) \colon \mathrm{Tan}(\mathrm{Rep}(H_1,c_1)) \rightarrow \mathrm{Tan}(\mathrm{Rep}(H_2,c_2))
	\end{align*}		
	by $r$ and then have to show that $r\gamma_{(H_1,c_1)} = \gamma_{(H_2,c_2)}\Phi$. Let $x \in H_1$ and $\pi \in \mathscr{T}(c_2^*\hat{H_2})$. Observing that $\tau \circ \Phi \in \mathscr{T}(\hat{H_1})$ for every $\tau \in \mathscr{T}(\hat{H_2})$, we obtain $(c_1^*)^{-1}(\pi) = (c_2^*)^{-1}(\pi) \circ \Phi$ and thus
		\begin{align*}
			(r\gamma_{(H_1,c_1)}(x))_\pi &= ((c_1^*)^{-1}(\pi))_{x} = ((c_2^*)^{-1}(\pi))_{\Phi(x)} = (\gamma_{(H_2,c_2)}(\Phi(x)))_{\pi}  \\
			&= ((\gamma_{(H_2,c_2)}\Phi)(x))_{\pi}.
		\end{align*}
	This shows the desired equality.
\end{proof}

\begin{proposition}\label{natisotannaka2}
	For a grouplike subset $\sigma \subset \hat{G}$ we obtain $\sigma = \mathrm{Rep}(\mathrm{Tan}(\sigma))$. In particular, the identity on every grouplike subset defines a natural isomorphism 
		\begin{align*}
			\gamma \colon \mathrm{Id}_{\mathbf{GrpLike}(\hat{G})^{\mathrm{op}}} \rightarrow \mathrm{Rep} \circ \mathrm{Tan}.
		\end{align*}
\end{proposition}
\begin{proof}
	This is a direct consequence of \cref{inversebij}.
\end{proof}

Combining \cref{natisotannaka1} and \cref{natisotannaka2}, we obtain the following Tannaka-Krein duality theorem for group compactifications of a general topological group $G$.

\begin{theorem}\label{grpcompvsgrplike}
	The functors 
		\begin{align*}
			\mathrm{Rep}&\colon \mathbf{Comp}(G) \rightarrow \mathbf{GrpLike}(\hat{G})^{\mathrm{op}}, \\
			\mathrm{Tan}&\colon  \mathbf{GrpLike}(\hat{G})^{\mathrm{op}} \rightarrow \mathbf{Comp}(G)
		\end{align*}
	establish an equivalence between the category of group compactifications and the opposite category of grouplike subsets.
\end{theorem}

\begin{remark}\label{translationproperties}
	Many properties of a group compactification $(H,c)$ translate directly into properties of its dual $\hat{H}$ or (equivalently) the associated grouplike subset $c^*(\hat{H}) \subset \hat{G}$.
		\begin{enumerate}[(i)]
			\item $H$ is finite if and only if $\hat{H}$ is finite (see \cite[Lemma 28.1]{HeRo1970}).
			\item $H$ is second countable if and only if $\hat{H}$ is countable (see \cite[Theorem 28.2]{HeRo1970}).
			\item $H$ is connected if and only if $\hat{H}$ is torsion-free, i.e., the grouplike subset generated by each $[\pi]\in \hat{H}\setminus\{[\mathbbm{1}]\}$ is infinite (see \cite[Theorem 28.21]{HeRo1970}).
			\item $H$ is totally disconnected if and only if $\hat{H}$ consists only of torsion elements, i.e., the grouplike subset generated by each $[\pi]\in \hat{H}$ is finite (see \cite[Theorem 28.19]{HeRo1970}).
			\item $H$ is a Lie group if and only if $\hat{H}$ is finitely generated, i.e, there are $[\pi_1], \dots, [\pi_n] \in \hat{H}$ such that the generated grouplike subset by these elements is the whole dual $\hat{H}$.\footnote{This is an easy consequence of the fact that a compact group is a Lie group if and only if it admits an injective finite-dimensional representation (see, e.g., \cite[Corollary 2.40 and Definition 2.41]{HoMo2006} or \cite[Theorem 5.13]{Foll2016}) combined with \cref{pointseparating}.}
			\item $H$ is abelian, if and only if $\dim([\pi]) = 1$ for every $[\pi] \in \hat{H}$ (see \cite[Remark 27.51]{HeRo1970}).
		\end{enumerate}
\end{remark}
\subsection{Normal dynamical systems}\label{scatgen5}
We now return to irreducible topological and measure-preserving systems with discrete spectrum. Which of these systems are isomorphic to a group rotation system? And what is the desired complete isomorphism invariant for such systems? To answer these questions, we define a generalized version of the point spectrum of a bounded strongly continuous representation of $G$ from \cref{defpointspect} by considering all irreducible finite-dimensional subrepresentations.

\begin{definition}\label{defmultpspec}
	Let $T\colon G \rightarrow \mathscr{L}(E)$ be a bounded strongly continuous group representation of $G$ on a Banach space $E$. 
		\begin{enumerate}[(i)]
			\item For $[\pi] \in \hat{G}$ we call 
				\begin{align*}
					\, \, \, \, \quad \quad \mathrm{mult}(T;[\pi])&\coloneqq \sup \left\{n \in \N_0 \mmid \bigoplus_{k=1}^n \pi \textrm{ is isomorphic to a subrepresentation of } T\right\}
				\end{align*}
			the \emph{multiplicity of $[\pi]$ in $T$}.
			\item The set
				\begin{align*}
					\sigma_{\mathrm{p}}(T) \coloneqq \{[\pi] \in \hat{G}\mid \mathrm{mult}(T;[\pi]) > 0\} 
				\end{align*}
				is the \emph{point spectrum of $T$}.
		\end{enumerate}		
	For a topological dynamical system $(K,\varphi)$ we set $\mathrm{mult}(K,\varphi;[\pi]) \coloneqq \mathrm{mult}(T^\varphi;[\pi])$ for $[\pi] \in \hat{G}$, and $\sigma_{\mathrm{p}}(K,\varphi) \coloneqq \sigma_{\mathrm{p}}(T^\varphi)$. Likewise, if $(\uX,T)$ is a measure-preserving system, we write $\mathrm{mult}(\uX,T;[\pi]) \coloneqq \mathrm{mult}(T;[\pi])$ for $[\pi] \in \hat{G}$, and $\sigma_{\mathrm{p}}(\uX,T) \coloneqq \sigma_{\mathrm{p}}(T)$. 
\end{definition}

\begin{remark}
	Note that for abelian $G$ the notions of the point spectrum from \cref{defpointspect} and \cref{defmultpspec} (ii) agree up to the canonical identification of $G^*$ with $\hat{G}$ from \cref{abeliandual}.
\end{remark}

Based on \cite[Corollary 5.8]{Zimm1976} we now introduce normal representations of $T$ via the following \enquote{spectral conditions}.
\begin{definition}\label{defnormal}
	A bounded strongly continuous representation $T\colon G \rightarrow \mathscr{L}(E)$ of $G$ with discrete spectrum  on a Banach space $E$ is \emph{normal} if 
		\begin{enumerate}[(i)]
			\item $\sigma_{\mathrm{p}}(T)$ is grouplike, and
			\item $\mathrm{mult}(T;[\pi]) \in \{0, \dim([\pi])\}$ for every $[\pi] \in \hat{G}$.
		\end{enumerate}
		A topological dynamical system $(K,\varphi)$ is \emph{normal} if the induced Koopman representation $T^\varphi$ is normal. Likewise, a measure-preserving system $(\uX,T)$ is \emph{normal} if the representation $T$ is normal.
\end{definition}
\begin{remark}
	\begin{enumerate}[(i)]
		\item Notice that a normal representation $T$ of $G$ automatically satisfies that its fixed space 
		\begin{align*}
			\fix(T) \coloneqq \{v \in E \mid T_tv = v \textrm{ for every } t \in G\}
		\end{align*}
	has dimension zero or one. In particular, every normal topological dynamical system is minimal (see \cref{charirred}), and every normal measure-preserving system is ergodic.
		\item The point spectrum $\sigma_{\mathrm{p}}(K,\varphi)$ of a topological dynamical system automatically contains $[\mathbbm{1}]$ and is closed with respect to taking complex conjugation. Thus, $\sigma_{\mathrm{p}}(K,\varphi)$ is grouplike if and only if $[\pi] \otimes [\varrho] \subset \sigma_{\mathrm{p}}(K,\varphi)$ for all $[\pi], [\varrho] \in \sigma_{\mathrm{p}}(K,\varphi)$. An analoguous statement holds for measure-preserving systems.
		\item In his situation, Zimmer gives another definition of normal measure-preserving systems based on virtual subgroups and measurable Hilbert bundles (see \cite[Section 5]{Zimm1976}). 
	\end{enumerate}
\end{remark}

\begin{example}\label{protonormal}
	Let $(H,c)$ be a group compatification of $G$. Then the induced topological and measure-preserving dynamical systems $(H,\varphi_c)$ and $(H,m_H,T^{\varphi_c})$ (see \cref{topexample} and \cref{measrot}) are normal by the Peter-Weyl theorem (see \cite[Theorem 5.12]{Foll2016}).
\end{example}
For an abelian group $G$, it is easily seen that every irreducible topological and measure-preserving system with discrete spectrum is normal (cf. \cref{subgroup1} and \cite[Theorem 4.21]{EFHN2015}). In the general case, we obtain the following result on normality of topological dynamical systems (using the notation of \cref{scatcomm}).

\begin{proposition}\label{normaltop}
	For a minimal topological dynamical system $(K,\varphi)$ with discrete spectrum the inequality $\mathrm{mult}(K,\varphi;[\pi]) \leq \dim([\pi])$ holds for every $[\pi] \in \hat{G}$. Moreover, the following assertions are equivalent.
		\begin{enumerate}[(a)]
			\item The system $(K,\varphi)$ is normal.
			\item If $(t_i)_{i \in I}$ is a net in $G$ with $\lim_{i} \varphi_{t_i}(x) =x$ for one $x \in K$, then $\lim_i \varphi_{t_i}(x) = x$ for every $x \in K$.
			\item For every $y \in K$ the natural extension
				\begin{align*}
					(\mathrm{E}_{\mathrm{u}}(K,\varphi),\varphi_{i_{(K,\varphi)}}) \rightarrow (K,\varphi), \quad \vartheta \mapsto \vartheta(y)
				\end{align*}
				is an isomorphism.
		\end{enumerate}
\end{proposition}
For the proof of \cref{normaltop} we use the following lemma. 
\begin{lemma}\label{generatedgrplike}
	Let $(K,\varphi)$ be a minimal topological dynamical system with discrete spectrum. Let $\langle\sigma_{\mathrm{p}}(K,\varphi)\rangle$ be the smallest grouplike subset of $\hat{G}$ containing $\sigma_{\mathrm{p}}(K,\varphi)$. Then
		\begin{align*}
			i_{(K,\varphi)}^* \colon \hat{\mathrm{E}_{\mathrm{u}}(K,\varphi)} \rightarrow \langle\sigma_{\mathrm{p}}(K,\varphi)\rangle, \quad [\pi] \mapsto [\pi \circ i_{(K,\varphi)}]
		\end{align*}
	is a bijection.
\end{lemma}
\begin{proof}
	It is clear that 
		\begin{align*}
			i_{(K,\varphi)}^* \colon \hat{\mathrm{E}_{\mathrm{u}}(K,\varphi)} \rightarrow \hat{G}, \quad [\pi] \mapsto [\pi \circ i_{(K,\varphi)}]
		\end{align*}
	is a well-defined and injective map. Moreover, its image is grouplike. Now consider the grouplike subset
		\begin{align*}
			\sigma \coloneqq  (i_{(K,\varphi)}^*)^{-1}(\langle\sigma_{\mathrm{p}}(K,\varphi)\rangle) 
		\end{align*}
	of $\hat{\mathrm{E}_{\mathrm{u}}(K,\varphi)}$. We show that $\sigma$ separates the points of $\mathrm{E}_{\mathrm{u}}(K,\varphi)$. Take $\vartheta \in \mathrm{E}_{\mathrm{u}}(K,\varphi)\setminus \{\mathrm{id}_K\}$. Since $(K,\varphi)$ has discrete spectrum, we find a finite-dimensional invariant subspace $M \subset \mathrm{C}(K)$ of the Koopman representation $T^\varphi$ and $f \in M$ with $f \circ \vartheta^{-1} \neq f$. We may assume that the subrepresentation $\pi_M \coloneqq (T^\varphi)|_M$ of $G$ is irreducible. Then the induced representation
		\begin{align*}
			\pi \colon \mathrm{E}_{\mathrm{u}}(K,\varphi) \rightarrow \mathscr{L}(M), \quad \omega \mapsto [f \mapsto f \circ \omega^{-1}]
		\end{align*}	
	is also irreducible and satisfies $\pi(\vartheta) \neq \mathrm{id}_M$. Moreover, $\pi \circ i_{(K,\varphi)} = \pi_M$, in particular, $[\pi] \in \sigma$. By \cref{pointseparating} we obtain $\sigma = \hat{\mathrm{E}_{\mathrm{u}}(K,\varphi)}$ and thus $i_{(K,\varphi)}^*$ is surjective.
\end{proof}
\begin{proof}[Proof of \cref{normaltop}.]
	Pick a point $y \in K$. The corresponding extension 
		\begin{align*}
			q \colon (\mathrm{E}_{\mathrm{u}}(K,\varphi),\varphi_{i_{(K,\varphi)}}) \rightarrow (K,\varphi), \quad \vartheta \mapsto \vartheta(y)
		\end{align*}
	induces an embedding of representations $J_q \colon \mathrm{C}(K) \rightarrow \mathrm{C}(\mathrm{E}_{\mathrm{u}}(K,\varphi))$, see \cref{koopmanfunct}. In view of \cref{protonormal}, we therefore obtain 
		\begin{align*}
			\mathrm{mult}(K,\varphi;[\pi]) \leq \mathrm{mult}(\mathrm{E}_{\mathrm{u}}(K,\varphi),\varphi_{i_{(K,\varphi)}};[\pi]) \leq \dim([\pi])
		\end{align*}
	for every $[\pi] \in \hat{G}$, as desired.
	
	We show that statements (a) -- (c) are equivalent. The implication \enquote{(c) $\Rightarrow$ (a)} is clear (see \cref{protonormal}). Likewise, the implication \enquote{(c) $\Rightarrow$ (b)} is obvious.
	
	Now assume that (a) holds. Given $[\pi] \in \hat{\mathrm{E}_{\mathrm{u}}(K,\varphi)}$, let $M_{[\pi]}$ be the sum of all subspaces of $\mathrm{C}(\mathrm{E}_{\mathrm{u}}(K,\varphi))$ on which the left-regular representation of $\mathrm{E}_{\mathrm{u}}(K,\varphi)$ is isomorphic to $\pi$. By \cref{generatedgrplike} and \cref{defnormal}, we obtain 
		\begin{enumerate}[(i)]
			\item $i_{(K,\varphi)}^*([\pi]) \in \sigma_{\mathrm{p}}(K,\varphi)$, and 
			\item $\mathrm{mult}(K,\varphi;i_{(K,\varphi)}^*([\pi])) = \dim([\pi])$.
		\end{enumerate}
	By (i) we obtain that $\pi$ is isomorphic to a finite-dimensional irreducible subrepresentation of
		\begin{align*}
			S \colon \mathrm{E}_{\mathrm{u}}(K,\varphi) \rightarrow \mathscr{L}(\mathrm{C}(K)), \quad \vartheta \mapsto [f \mapsto f \circ \vartheta^{-1}],
		\end{align*}
	and (ii) then implies $\mathrm{mult}([\pi];S) = \dim([\pi])$. This shows
		\begin{align*}
			\dim(\rg(J_q) \cap M_{[\pi]}) \geq \dim([\pi])^2.
		\end{align*}
	 On the other hand, $\dim(M_{[\pi]}) = \dim([\pi])^2$ by the Peter-Weyl theorem (see \cite[Theorem 5.12]{Foll2016}) and hence $\rg(J_q)$ contains $M_{[\pi]}$. Since the subspaces $M_{[\pi]}$ for $[\pi] \in \hat{\mathrm{E}_{\mathrm{u}}(K,\varphi)}$ span a dense subspace of $\mathrm{C}(\mathrm{E}_{\mathrm{u}}(K,\varphi))$ (see again \cite[Theorem 5.12]{Foll2016}), we conclude that $J_q$ is surjective. Thus, $q$ is an isomorphism, and this yields (c).
	
	Finally, assume (b). For $\vartheta_1,\vartheta_2 \in \mathrm{E}_{\mathrm{u}}(K,\varphi)$ with $\vartheta_1(y) = \vartheta_2(y)$ set $\vartheta \coloneqq \vartheta_2^{-1} \circ \vartheta_1 \in \mathrm{E}_{\mathrm{u}}(K,\varphi)$ and choose a net $(\vartheta_{t_i})_{i \in I}$ with $\vartheta = \lim_{i} \vartheta_{t_i}$. Then $\lim_{i}\vartheta_i(y) = \vartheta(y) =y$ and hence $\vartheta(x) = \lim_i \vartheta_i(x) = x$ for every $x \in K$ by (b). This shows $\vartheta_1 = \vartheta_2$ and therefore $q$ is injective. Thus, (b) implies (c).
\end{proof}
Denoting the category of pointed normal topological dynamical systems by $\mathbf{TopNorm}_{\mathbf{pt}}(G)$, \cref{normaltop} yields the following analogue of \cref{natiso1comp}.
\begin{proposition}\label{natiso1norm}
	For every pointed normal topological dynamical system $(K,\varphi,y)$ we obtain an isomorphism $\gamma_{(K,\varphi,y)} \colon \mathrm{Rot}(\mathrm{Env}(K,\varphi,y)) \rightarrow (K,\varphi,y), \, \vartheta \mapsto \vartheta(y)$. Moreover, these isomorphisms define a natural isomorphism 
		\begin{align*}
			\gamma \colon \mathrm{Rot} \circ \mathrm{Env} \rightarrow \mathrm{Id}_{\mathbf{TopNorm}_{\mathbf{pt}}(G)}.
		\end{align*}
\end{proposition}
As a consequence of \cref{natiso1norm}, we obtain the following generalization of \cref{pointedvsnonpointed2} by the same proof.
\begin{corollary}\label{pointedvsnonpointed3}
	Let $(K,\varphi,y)$ and $(L,\psi,z)$ be pointed normal topological dynamical systems. Then $(K,\varphi,y)$ and $(L,\psi,z)$ are isomorphic if and only if $(K,\varphi)$ and $(L,\psi)$ are isomorphic.
\end{corollary}
We also obtain that the restriction of the enveloping group functor from \cref{envgrpfunct}
	\begin{align*}
		\mathrm{Env}\colon \mathbf{TopNorm}_{\mathbf{pt}}(G) \rightarrow \mathbf{Comp}(G)
	\end{align*}
defines an essential inverse of the rotation functor 
	\begin{align*}
		\mathrm{Rot} \colon \mathbf{Comp}(G) \rightarrow \mathbf{TopNorm}_{\mathbf{pt}}(G)
	\end{align*}
from \cref{rotfunct}. 
In combination with \cref{natiso2comp} we obtain a categorical equivalence (cf. \cref{equivcomp}).
\begin{theorem}\label{equivcomp2}
	The functors
		\begin{align*}
			\mathrm{Env}&\colon \mathbf{TopNorm}_{\mathbf{pt}}(G) \rightarrow \mathbf{Comp}(G),\\
			\mathrm{Rot} &\colon \mathbf{Comp}(G) \rightarrow \mathbf{TopNorm}_{\mathbf{pt}}(G)
		\end{align*}
	establish an equivalence between the category of pointed normal topological dynamical systems and the category of group compactifications.
\end{theorem}

We also note that \cref{equivtopmeas} implies that the categories $\mathbf{TopNorm}(G)$ of normal topological dynamical systems and the opposite category $\mathbf{MeasNorm}(G)^{\mathrm{op}}$ of normal measure-preserving systems are equivalent.
\begin{theorem}\label{normalequiv}
	The restricted functors 
		\begin{align*}
			\mathrm{Meas}&\colon \mathbf{TopNorm}(G) \rightarrow \mathbf{MeasNorm}(G)^{\mathrm{op}},\\
			\mathrm{Top}&\colon  \mathbf{MeasNorm}(G)^{\mathrm{op}} \rightarrow \mathbf{TopNorm}(G)
		\end{align*}
	establish an equivalence between the category of normal topological dynamical systems and the opposite category of normal measure-preserving systems.
\end{theorem}

\subsection{A Halmos-von Neumann theorem for normal actions}\label{scatgen6}
	We now state versions of the Halmos-von Neumann theorem for normal dynamical systems, first in a categorical language, and then in a more classical formulation. 
	
	\begin{construction}
		We define a functor $\sigma_{\mathrm{p}} \colon \mathbf{TopNorm}_{\mathbf{pt}}(G) \rightarrow \mathbf{GrpLike}(\hat{G})^{\mathrm{op}}$ by setting
			\begin{itemize}
				\item $\sigma_{\mathrm{p}}(K,\varphi,y) \coloneqq \sigma_{\mathrm{p}}(K,\varphi)$ for every pointed normal topological dynamical system $(K,\varphi,y)$, and
				\item $\sigma_{\mathrm{p}}(q) \coloneqq i_{\sigma_{\mathrm{p}}(L,\psi),\sigma_{\mathrm{p}}(K,\varphi)}$ for every morphism $q \colon (K,\varphi,y) \rightarrow (L,\psi,z)$.
			\end{itemize}
		It is called the \emph{point spectrum functor}.
	\end{construction}
	
	With the enveloping group functor $\mathrm{Env}$ from \cref{envgrpfunct} and the representations functor $\mathrm{Rep}$ from \cref{represfunct}, \cref{generatedgrplike} implies the following.
	\begin{proposition}\label{functorcomm2}
		The diagram of functors
			\[
				\xymatrix{
					\mathbf{TopNorm}_{\mathbf{pt}}(G) \ar[rr]^{\sigma_{\mathrm{p}}} \ar[rd]_{\mathrm{Env}} & & \mathbf{GrpLike}(\hat{G})^{\mathrm{op}} \\
					& \mathbf{Comp}(G) \ar[ru]_{\mathrm{Rep}} &  
				}
			\]	
		commutes.
	\end{proposition}
	
	Combining \cref{functorcomm2} with \cref{grpcompvsgrplike} and \cref{equivcomp2}, we obtain the following extension of \cref{hvncat1} to the setting of a general topological group $G$. 
	\begin{theorem}\label{main2}
		\begin{enumerate}[(i)]
				\item The functors
					\begin{align*}
						\mathrm{Env} &\colon\mathbf{TopNorm}_{\mathbf{pt}}(G)\rightarrow \mathbf{Comp}(G),\\
						\mathrm{Rot} &\colon  \mathbf{Comp}(G) \rightarrow \mathbf{TopNorm}_{\mathbf{pt}}(G)
					\end{align*}
				establish an equivalence between the category of pointed normal topological dynamical systems and the category of group compactifications.
				\item The functors
					\begin{align*}
						\mathrm{Rep} &\colon \mathbf{Comp}(G) \rightarrow \mathbf{GrpLike}(\hat{G})^{\mathrm{op}},\\
						\mathrm{Tan} &\colon \mathbf{GrpLike}(\hat{G})^{\mathrm{op}} \rightarrow \mathbf{Comp}(G) 
					\end{align*}
				establish an equivalence between the category of group compactifications and the opposite category of grouplike subsets of $\hat{G}$.
				\item The functor 
					\begin{align*}
						\sigma_{\mathrm{p}} = \mathrm{Rep} \circ \mathrm{Env} \colon \mathbf{TopNorm}_{\mathbf{pt}}(G) \rightarrow \mathbf{GrpLike}(\hat{G})^{\mathrm{op}}
					\end{align*}
				defines an equivalence between the category of pointed normal topological dynamical systems and the opposite category of grouplike subsets of $\hat{G}$.
			\end{enumerate}
	\end{theorem}
	
	We formulate a slightly weaker \enquote{category-free} version which is the analogue of \cref{comhvn1}. In view of \cref{pointedvsnonpointed3}, it is a direct consequence of \cref{main2}.
	\begin{theorem}\label{nomralhvntop}
		For normal topological dynamical systems the following assertions hold.
	\begin{enumerate}[(i)]
    	\item Two normal systems $(K_1,\varphi_1)$ and $(K_2,\varphi_2)$ are isomorphic if and only if $\sigma_{\mathrm{p}}(K_1,\varphi_1) = \sigma_{\mathrm{p}}(K_2,\varphi_2)$. 
    	\item For every grouplike subset $\sigma \subset \hat{G}$ there is a normal system $(K,\varphi)$ such that $\sigma_{\mathrm{p}}(K,\varphi) = \sigma$.
    	\item For every normal system $(K,\varphi)$ there is a group compactification $(H,c)$ such that $(K,\varphi)$ is isomorphic to the rotation $(H,\varphi_c)$.
	\end{enumerate}
	\end{theorem}
	
	Using the categorical equivalence between normal topological and measure-preserving dynamical systems from \cref{normalequiv}, we obtain the ergodic theoretical variant of \cref{nomralhvntop}. It is a version of Zimmer's results in the case of systems (see \cite[Theorems 4.3 and 5.7, Theorem 6.2, and Theorem 6.4]{Zimm1976}) for general normal measure-preserving systems dropping any separability conditions on the involved probability spaces and  on the group $G$ (as well as local compactness of $G$).
	\begin{theorem}\label{nomralhvntop2}
		For normal systems measure-preserving systems the following assertions hold.
	\begin{enumerate}[(i)]
    	\item Two normal systems $(\uX_1,T_1)$ and $(\uX_2,T_2)$ are isomorphic if and only if $\sigma_{\mathrm{p}}(\uX_1,T_1) = \sigma_{\mathrm{p}}(\uX_2,T_2)$. 
    	\item For every grouplike subset $\sigma \subset \hat{G}$ there is a normal system $(\uX,T)$ such that $\sigma_{\mathrm{p}}(\uX,T) = \sigma$.
    	\item For every normal system $(\uX,T)$ there is a group compactification $(H,c)$ such that $(\uX,T)$ is isomorphic to the rotation $(H,m_H,T^{\varphi_c})$.
	\end{enumerate}
	\end{theorem}
	\begin{remark}\label{finalrem}
		\begin{enumerate}[(i)]
			\item As a consequence of \cref{main2}, we obtain for normal topological dynamical systems $(K_1,\varphi_1)$ and $(K_2,\varphi_2)$ that the existence of a morphism $q \colon (K_1,\varphi_1) \rightarrow (K_2,\varphi_2)$ is equivalent to the existence of a (not-necessarily continuous) injective linear map $V \colon \mathrm{C}(K_2) \rightarrow \mathrm{C}(K_1)$ intertwining the Koopman representations $T^{\varphi_2}$ and $T^{\varphi_1}$. Moreover, $(K_1,\varphi_1)$ and $(K_2,\varphi_2)$ are isomorphic if and only we can find such $V$ which is a vector space isomorphism. An analoguous statement holds in the ergodic theoretical setting. In particular, two normal measure-preserving systems $(\uX_1,T_1)$ and $(\uX_2,T_2)$ are isomorphic if and only the representations $T_1$ and $T_2$ restricted to the corresponding $\mathrm{L}^2$-spaces are unitarily equivalent (cf. \cite[Corollary 17.14]{EFHN2015}).
			\item In view of \cref{translationproperties}, many properties of a normal system (or the corresponding group compactification) translate directly into properties of its point spectrum and vice versa.
			\item A general (possibly non-normal) minimal topological or ergodic measure-preserving system with discrete spectrum is always isomorphic to a quasi-rotation (in the topological case, consider the map \cref{normaltop} and factor out the stabilizer group; for the measure-preserving case use the categorical equivalence of \cref{equivtopmeas}). This representation aspect of the Halmos-von Neumann result is well-known (see, e.g., \cite[Theorem 1]{Mack1964}, \cite[Theorem 3.6]{Ausl1988}, and \cite[Section 3]{NaWo1972}). However there is an obstacle to covering the uniqueness and realization aspects of \cref{nomralhvntop2} for general irreducible systems with discrete spectrum already pointed out by Mackey (see \cite[Remark 2 of Section 3]{Mack1964}): Due to an example of Todd (see \cite{Todd1950}) there are non-isomorphic ergodic systems $(\uX_1,T_1)$ and $(\uX_2,T_2)$ with discrete spectrum such that the representations $T_1$ and $T_2$ restricted to the corresponding $\mathrm{L}^2$-spaces are unitarily equivalent\footnote{In fact, the acting group $G$ can even be chosen to be finite and the systems $(\uX_1,T_1)$ and $(\uX_2,T_2)$ are then given by quasi-rotations on homogeneous spaces of $G$.}. Thus, any isomorphism invariant depending only on these unitary representations (such as the point spectrum from \cref{defmultpspec}), cannot be complete. Notice however, that in view of \cref{generatedgrplike} and \cref{grpcompvsgrplike} two minimal systems with discrete spectrum $(K_1,\varphi)$ and $(K_2,\varphi_2)$ with the same point spectrum induce isomorphic group compactifications $(\mathrm{E}_{\mathrm{u}}(K_1,\varphi_1),i_{K_1,\varphi_1)})$ and $(\mathrm{E}_{\mathrm{u}}(K_2,\varphi_2),i_{K_2,\varphi_2)})$. In particular, two irreducible topological or measure-preserving systems with the same point spectrum can be realized as two factors of the same normal system. We also refer to Mackey's article \cite{Mack1964} for an alternative approach to the Halmos-von Neumann theorem for general ergodic measure-preserving systems with discrete spectrum.
		\end{enumerate}
	\end{remark}
	
\parindent 0pt
\parskip 0.5\baselineskip
\setlength{\footskip}{4ex}
\bibliographystyle{alpha}
\footnotesize
\subsection*{Conflict of Interest}
The authors declare that they have no conflict of interest.
\subsection*{Availability of data and materials}
NA.
\bibliography{./bib/bibliography} 
\footnotesize

\end{document}